\documentclass{article}
\usepackage{amsmath,amssymb,amsthm,graphicx,color}
\usepackage{thmtools}
\usepackage{authblk}

\usepackage[colorlinks=true,linkcolor=blue,citecolor=blue,urlcolor=blue,pdfpagelabels=false]{hyperref}

\usepackage{geometry}
\makeatletter
\newcommand{\subjclass}[2][1991]{%
  \let\@oldtitle\@title%
  \gdef\@title{\@oldtitle\footnotetext{#1 \emph{Mathematics subject classification.} #2}}%
}
\newcommand{\keywords}[1]{%
  \let\@@oldtitle\@title%
  \gdef\@title{\@@oldtitle\footnotetext{\emph{Key words and phrases.} #1.}}%
}
\makeatother

\theoremstyle{plain}
  \declaretheorem[numberwithin=section]{theorem}
  \declaretheorem[numberlike=theorem]{corollary}
  \declaretheorem[numberlike=theorem]{proposition}
  \declaretheorem[numberlike=theorem]{lemma}
  
\theoremstyle{definition}
  \declaretheorem[numberlike=theorem,qed=$\diamond$]{example}
  \declaretheorem[numberlike=theorem]{remark}
\newenvironment{acknowledgements}{\bigskip\textbf{Acknowledgements.}}{}

\newcommand{\op}[1]{\ensuremath{\operatorname{#1}}}
\newcommand{\sinc}{\op{sinc}}
\newcommand{\mathd}{\mathrm{d}}
\newcommand{\dueto}[1]{\textup{\textbf{(#1) }}}

\newcommand{\pFq}[5]{\ensuremath{{}_{#1}F_{#2} \left( \genfrac{}{}{0pt}{}{#3}{#4} \middle| {#5} \right)}}
\newcommand{\MeijerG}[5]{\ensuremath{G_{#1}^{#2} \left( \genfrac{}{}{0pt}{}{#3}{#4} \middle| {#5} \right)}}

\newcommand{\assign}{:=}

\renewcommand{\ge}{\geqslant}
\renewcommand{\le}{\leqslant}
\renewcommand{\geq}{\geqslant}
\renewcommand{\leq}{\leqslant}

\renewcommand{\Re}{\op{Re}\,}

\begin{document}

\title{Densities of short uniform random walks in higher dimensions}

\author[1]{Jonathan M. Borwein}
\author[2]{Armin Straub}
\author[3]{Christophe Vignat}
\affil[1]{University of Newcastle}
\affil[2]{University of Illinois at Urbana-Champaign}
\affil[3]{Tulane University}




\date{August 19, 2015}

\maketitle

\begin{abstract}
We study arithmetic properties of short uniform random
walks in arbitrary dimensions, with a focus on explicit (hypergeometric)
evaluations of the moment functions and probability densities in the case of
up to five steps. Somewhat to our surprise, we are able to provide complete
extensions to arbitrary dimensions for most of the central results known in the two-dimensional
case.
\end{abstract}

\section{Introduction}

An $n$-step uniform random walk in $\mathbb{R}^d$ starts at the origin and
consists of $n$ independent steps of length $1$, each of which is taken into a
uniformly random direction. In other words, each step corresponds to a random
vector uniformly distributed on the unit sphere. The study of such walks
originates with Pearson \cite{Pea05}, who was interested in planar walks,
which he looked at \cite{Pea06} as migrations of, for instance, mosquitos
moving a step after each breeding cycle. Random walks in three dimensions (known as
random flights) seem first to have been studied in extenso by Rayleigh
\cite{rayleigh-flights}, and higher dimensions are touched upon in
\cite[{\textsection}13.48]{watson-bessel}.

This paper is a companion to \cite{bnsw-rw,bsw-rw2} and
\cite{bswz-densities}, which studied the analytic and number theoretic
behaviour of short uniform random walks in the plane (five steps or less). In
this work we revisit the same issues in higher dimensions. Somewhat to our
surprise, we are able to provide complete extensions for most of the central
results in the culminating paper \cite{bswz-densities}.

Throughout the paper, $n$ and $d$ are the positive integers corresponding to
the number of steps and the dimension of the random walk we are considering.
Moreover, we denote with $\nu$ the half-integer
\begin{equation}
  \nu = \frac{d}{2} - 1. \label{eq:nu}
\end{equation}
It turns out that most results are more naturally expressed in terms of this
parameter $\nu$, and so we denote, for instance, with $p_n (\nu ; x)$ the
probability density function of the distance to the origin after $n$ random
unit steps in $\mathbb{R}^d$. In Section~\ref{ssec:den}, we mostly follow
the account in \cite{hughes-rw} and develop the basic Bessel integral
representations for these densities beginning with Theorems~\ref{thm:bi} and
\ref{thm:pn:besselD}, which are central to our analysis. In particular, a
brief discussion of the (elementary) case of odd dimensions is included in
Section \ref{ssec:odd}.

In Section \ref{ssec:mom}, we turn to general results on the associated moment
functions
\begin{equation}
  W_n (\nu ; s) = \int_0^{\infty} x^s p_n (\nu ; x) \mathd x. \label{eq:Wn}
\end{equation}
In particular, we derive in Theorem~\ref{thm:W:even} a formula for the even
moments $W_n (\nu ; 2 k)$ as a multiple sum over the product of multinomial
coefficients. As a consequence, we observe another interpretation of the
Catalan numbers as the even moments of the distance after two random steps in
four dimensions, and realize, more generally, in Example~\ref{eg:W:values4d}
the moments in four dimensions in terms of powers of the Narayana triangular
matrix. We shall see that dimensions two and four are privileged in that all
even moments are integral only in those two dimensions.

In Section~\ref{sec:mom}, we turn to a detailed analysis of the moments of
short step walks: from two-step walks ({\textsection}\ref{ssec:mom2}) through
five steps ({\textsection}\ref{ssec:mom5}). For instance, we show in
\eqref{eq:W2:gf}, Theorem~\ref{thm:W3:gf} and Example~\ref{eg:W4:ogf} that the
ordinary generating functions of the even moments for $n = 2, 3, 4$ can be
expressed in terms of hypergeometric functions. We are also able to give
closed forms for all odd moments for less than five steps.

In Section~\ref{sec:den}, we perform a corresponding analysis of the densities
$p_n (\nu ; x)$: from two-step walks ({\textsection}\ref{ssec:den2}) through
five steps ({\textsection}\ref{ssec:den5}). One especially striking result for
$n = 3$, shown in Corollary~\ref{cor:fe3}, is the following functional
equation for the probability density function $p_3 (\nu ; x)$. For $0
\leq x \leq 3$, and each half-integer $\nu \geq 0$, the
function $F (x) \assign p_3 (\nu ; x) / x$ satisfies the functional equation
\begin{equation}
  F (x) = \left( \frac{1 + x}{2} \right)^{6 \nu - 2} F \left( \frac{3 - x}{1 +
  x} \right) .
\end{equation}
Finally, in Section~\ref{sec:conc}, we make some concluding remarks and leave
several open questions.

As much as possible, we keep our notation consistent with that in
\cite{bnsw-rw,bsw-rw2}, and especially \cite{bswz-densities}, to which we
refer for details of how to exploit the Mellin transform and similar matters.
Random walks in higher dimensions are also briefly discussed in
\cite[Chapter~4]{wan-phd}. In particular, Wan gives evaluations in arbitrary
dimensions for the second moments, which we consider and generalize in
Example~\ref{eg:W:smallk}, as well as for two-step walks.

\section{Basic results from probability}\label{sec:basic}

\subsection{The probability densities}\label{ssec:den}

For the benefit of the reader, we briefly summarize the account given in
\cite[Chapter~2.2]{hughes-rw} of how to determine the probability density
function $p_n (\nu ; x)$ of an $n$-step random walk in $d$ dimensions. The
reader interested in further details and corresponding results for more
general random walks, for instance, with varying step sizes, is referred to
\cite{hughes-rw}.

Throughout the paper, the {\emph{normalized Bessel function of the first
kind}} is
\begin{equation}
  j_{\nu} ( x) = \nu ! \left( \frac{2}{x} \right)^{\nu} J_{\nu} ( x) = \nu !
  \sum_{m \geq 0} \frac{( - x^2 / 4)^m}{m! ( m + \nu) !} \label{eq:j} .
\end{equation}
With this normalization, we have $j_{\nu} (0) = 1$ and
\begin{equation}
  j_{\nu} (x) \sim \frac{\nu !}{\sqrt{\pi}} \left( \frac{2}{x} \right)^{\nu +
  1 / 2} \cos \left( x - \frac{\pi}{2}  \left( \nu + \frac{1}{2} \right)
  \right) \label{eq:j:infty}
\end{equation}
as $x \rightarrow \infty$ on the real line. Note also that $j_{1 / 2} (x) =
\operatorname{sinc} (x) = \sin (x) / x$, which in part explains why analysis in
3-space is so simple. More generally, all half-integer order $j_{\nu} (x)$ are
elementary.

\begin{theorem}
  {\dueto{Bessel integral for the densities, I}} \label{thm:bi}
  The probability density
  function of the distance to the origin in $d \geq 2$ dimensions after
  $n \geq 2$ steps is, for $x > 0$,
  \begin{equation}
    p_n (\nu ; x) = \frac{2^{- \nu}}{\nu !} \int_0^{\infty} (t x)^{\nu + 1}
    J_{\nu} (t x) j_{\nu}^n (t) \mathd t, \label{eq:pn:bessel}
  \end{equation}
  where, as introduced in \eqref{eq:nu}, $\nu = \frac{d}{2} - 1$.
\end{theorem}

\begin{proof}
  Let $\boldsymbol{X}$ be a random vector, which is uniformly distributed on the
  unit sphere in $\mathbb{R}^d$. That is, $\boldsymbol{X}$ describes the
  displacement of a single step in our random walk. Then the Fourier transform
  of its induced probability measure $\mu_{\boldsymbol{X}}= \frac{\Gamma\left(d/2\right)}{2 \pi ^ {d/2}} \delta\left(\Vert \boldsymbol{x} \Vert -1 \right)$ is
  \begin{equation*}
    \hat{\mu}_{\boldsymbol{X}} (\boldsymbol{q}) = \int_{\mathbb{R}^d} e^{i
     \langle \boldsymbol{x}, \boldsymbol{q} \rangle} \mathd \mu_{\boldsymbol{X}}
     (\boldsymbol{x}) = j_{\nu} (\|\boldsymbol{q}\|) .
  \end{equation*}
  This is a special case of the famous formula
  \cite[(2.30)]{hughes-rw},
  \begin{equation}
    \int_{\mathbb{R}^d} e^{i \langle \boldsymbol{x}, \boldsymbol{q} \rangle} f
    (\|\boldsymbol{x}\|) \mathd \boldsymbol{x}= 2 \pi^{d / 2} \int_0^{\infty}
    \left( \frac{2}{t q} \right)^{d / 2 - 1} J_{d / 2 - 1} (t q) t^{d - 1} f
    (t) \mathd t, \label{eq:ft:iso}
  \end{equation}
  with $q = \|\boldsymbol{q}\|$, for integrals of orthogonally invariant
  functions.

  Note that the position $\boldsymbol{Z}$ of a random walk after $n$ unit steps
  in $\mathbb{R}^d$ is distributed like the sum of $n$ independent copies of
  $\boldsymbol{X}$. The Fourier transform of $\mu_{\boldsymbol{Z}}$ therefore is
  \begin{equation}
    \hat{\mu}_{\boldsymbol{Z}} (\boldsymbol{q}) = j_{\nu}^n (\|\boldsymbol{q}\|).
    \label{eq:muz:hat}
  \end{equation}
  We are now able to obtain the probability density function $p_n (\nu ;
  \boldsymbol{x})$ of the position after $n$ unit steps in $\mathbb{R}^d$ via
  the inversion relation
  \begin{equation}
    p_n (\nu ; \boldsymbol{x}) = \frac{1}{(2 \pi)^d} \int_{\mathbb{R}^d} e^{- i
    \langle \boldsymbol{x}, \boldsymbol{q} \rangle} \hat{\mu}_{\boldsymbol{Z}}
    (\boldsymbol{q}) \mathd \boldsymbol{q}. \label{eq:pn:fi}
  \end{equation}
   Combining \eqref{eq:muz:hat}, \eqref{eq:pn:fi} and
  \eqref{eq:ft:iso}, we find
  \begin{equation*}
    p_n (\nu ; \boldsymbol{x}) = \frac{1}{(2 \pi)^{\nu + 1}} \int_0^{\infty}
     \frac{t^{\nu + 1}}{x^{\nu}} J_{\nu} (t x) j_{\nu}^n (t) \mathd t.
  \end{equation*}
  Since the surface area of the unit sphere in $\mathbb{R}^d$ is $2 \pi^{\nu
  + 1} / \nu !$, the density functions for the position and distance are
  related by
  \begin{equation}\label{eq:pn:posdis}
    p_n (\nu ; \boldsymbol{x}) = \frac{\nu !}{2 \pi^{\nu + 1}
     \|\boldsymbol{x}\|^{2 \nu + 1}} p_n (\nu ; \|\boldsymbol{x}\|),
  \end{equation}
  whence we arrive at the formula \eqref{eq:pn:bessel}.
\end{proof}

The probability densities $p_3 (\nu ; x)$ of the distance to the origin after
three random steps in dimensions $2, 3, \ldots, 9$ are depicted in
Figure~\ref{fig:p3}. With the exception of the planar case, which has a
logarithmic singularity at $x = 1$, these functions are at least continuous in
the interval $[0, 3]$, on which they are supported. Their precise regularity
is provided by Corollary~\ref{cor:pn:reg}.  For comparison, the probability
densities of four-step walks in dimensions $2, 3, \ldots, 9$ are plotted in
Figure~\ref{fig:p4}, and corresponding plots for five steps are provided by
Figure~\ref{fig:p5}.

\begin{figure}[h]
  \centering
  \includegraphics[width=9cm]{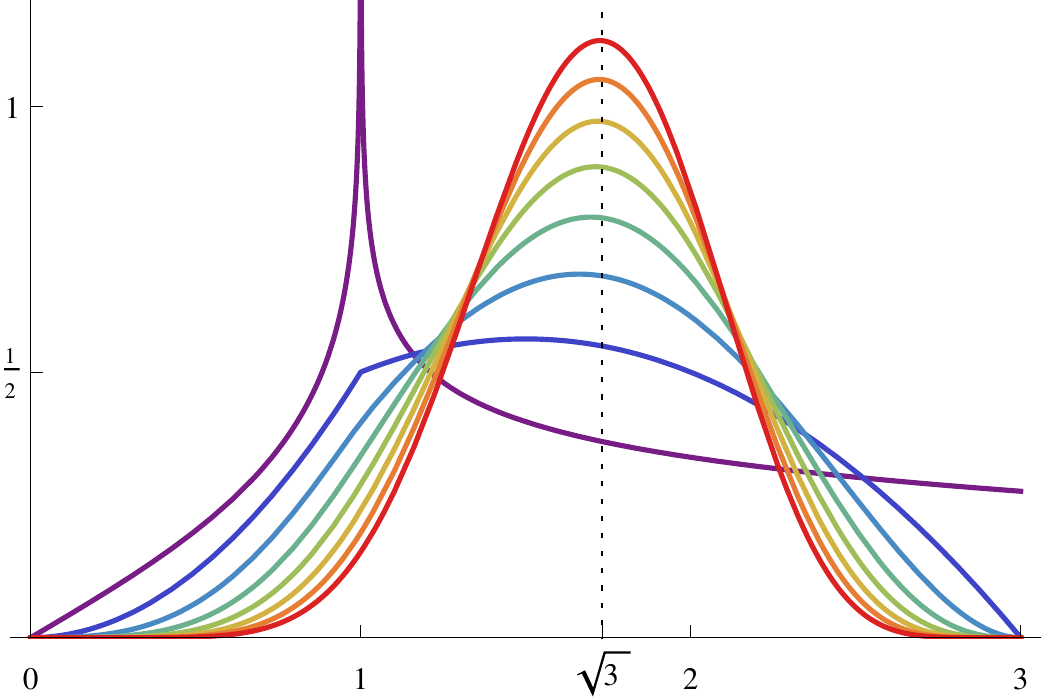}
  \caption{\label{fig:p3}$p_3 (\nu ; x)$ for $\nu = 0, \tfrac{1}{2}, 1,
  \ldots, \frac{7}{2}$}
\end{figure}

\begin{figure}[h]
  \centering
  \includegraphics[width=9cm]{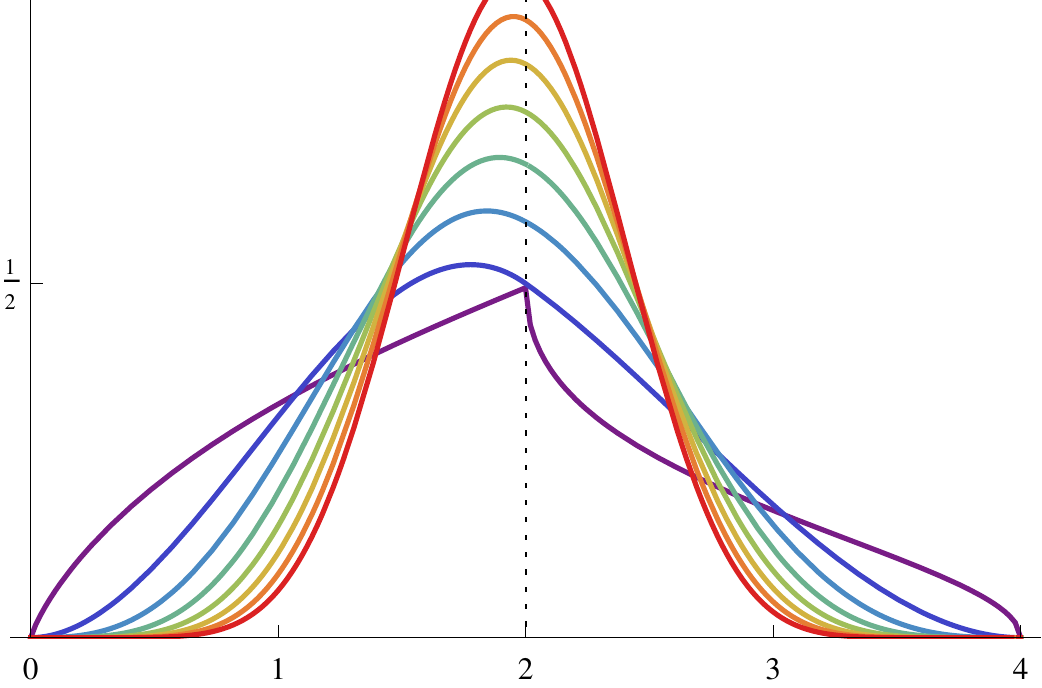}
  \caption{\label{fig:p4}$p_4 (\nu ; x)$ for $\nu = 0, \tfrac{1}{2}, 1,
  \ldots, \frac{7}{2}$}
\end{figure}

Observe how the density functions center and spike as the dimensions increase.
Indeed, this is a general phenomenon and the distributions described by the
densities $p_n (\nu ; x)$ approach a Dirac distribution centered at
$\sqrt{n}$. The intuition behind this observation is as follows: as the
dimension $d$ increases, the directions of each of the $n$ random steps
increasingly tend to be close to orthogonal to each other. That is, given an
incoming direction, the direction of the next step will probably belong to a
hyperplane which is almost orthogonal to this direction. Pythagoras' theorem
therefore predicts that the distance after $n$ steps is about $\sqrt{n}$. A
precise asymptotic result, which confirms this intuition, is given in
Example~\ref{eg:p:diminf}.

Integrating the Bessel integral representation \eqref{eq:pn:bessel} for the
probability density functions $p_n (\nu ; x)$, we obtain a corresponding
Bessel integral representation for the cumulative distribution functions,
\begin{equation}
  P_n (\nu ; x) = \int_0^x p_n (\nu ; y) \mathd y, \label{eq:Pn}
\end{equation}
of the distance to the origin after $n$ steps in $d$ dimensions.

\begin{corollary}\dueto{Cumulative distribution}
  Suppose $d \geq 2$ and $n \geq 2$. Then, for $x > 0$,
  \begin{equation}
    P_n (\nu ; x) = \frac{2^{- \nu}}{\nu !} \int_0^{\infty} (t x)^{\nu + 1}
    J_{\nu + 1} (t x) j_{\nu}^n (t) \frac{\mathd t}{t} . \label{eq:Pn:bessel}
  \end{equation}
\end{corollary}

\begin{example}\dueto{Kluyver's Theorem}
  A famous result of Kluyver \cite{Klu06} asserts that, for $n\geq2$,
  \begin{equation*}
    P_n (0 ; 1) = \frac{1}{n + 1} .
  \end{equation*}
  That is, after $n$ unit steps in the plane, the probability to be within one
  unit of the starting point is $1 / (n + 1)$. This is nearly immediate from \eqref{eq:Pn:bessel}.
  An elementary proof of this remarkable result was given recently by Bernardi
  \cite{bernardi-rw}.

  It is natural to wonder whether there exists an extension of this result to
  higher dimensions. Clearly, these probabilities decrease as the dimension
  increases.
  \begin{enumerate}
    \item In the case of two steps, we have
    \begin{equation*}
      P_2 (\nu ; x) = \frac{x^{2 \nu + 1}}{2 \sqrt{\pi}}  \frac{\Gamma (\nu +
       1)}{\Gamma (\nu + 3 / 2)} \pFq21{\tfrac{1}{2} + \nu, \tfrac{1}{2} - \nu}{\tfrac{3}{2} + \nu}{\frac{x^2}{4}},
    \end{equation*}
    which, in the case of integers $\nu \geq 0$ and $x = 1$, reduces to
    \begin{equation}
      P_2 (\nu ; 1) = \frac{1}{3} - \frac{\sqrt{3}}{4 \pi} \sum_{k = 0}^{\nu-1} \frac{3^k} {(2k+1){{2k }\choose k}} . \label{eq:P2:1}
    \end{equation}
    Alternatively,
     \begin{equation}
      P_2 (\nu ; 1) = \frac{1}{3} - \frac{1}{4 \pi} \sum_{k = 1}^{\nu} 3^{k -
      1 / 2} \frac{\Gamma (k)^2}{\Gamma (2 k)} . \label{eq:P2:1b}
    \end{equation}
    In particular, in dimensions $4$, $6$ and $8$,
    \begin{equation*}
      P_2 (1 ; 1) = \frac{1}{3} - \frac{\sqrt{3}}{4 \pi}, \hspace{1em} P_2 (2
       ; 1) = \frac{1}{3} - \frac{3 \sqrt{3}}{8 \pi}, \hspace{1em} P_2 (3 ; 1)
       = \frac{1}{3} - \frac{9 \sqrt{3}}{20 \pi}, \hspace{1em} \ldots,
    \end{equation*}
    and it is obvious from \eqref{eq:P2:1} that all the probabilities $P_2
    (\nu ; 1)$ are of the form $\frac{1}{3} - c_{\nu} \frac{\sqrt{3}}{\pi}$
    for some rational factor $c_{\nu} > 0$.

    \item In the case of three steps, we find
    \begin{equation*}
      P_3 (1 ; 1) = \frac{1}{4} - \frac{4}{3 \pi^2}, \hspace{1em} P_3 (2 ; 1)
       = \frac{1}{4} - \frac{256}{135 \pi^2}, \hspace{1em} P_3 (3 ; 1) =
       \frac{1}{4} - \frac{2048}{945 \pi^2}, \hspace{1em} \ldots
    \end{equation*}
    Indeed, for integers $\nu \geq 0$, we have the general formula
    \begin{equation}
      P_3 (\nu ; 1) = \frac{1}{4} - \frac{1}{3 \pi^2} \sum_{k = 1}^{\nu}
      2^{6 (k-1)} (11 k - 3) \frac{\Gamma^5 (k)}{\Gamma (2 k) \Gamma (3 k)} .
      \label{eq:P3:1}
    \end{equation}
    In the limit we arrive at the improbable evaluation
    \begin{equation}
      \pFq54{\tfrac{19}{11}, 1, 1, 1, 1}{\tfrac{8}{11}, \tfrac{4}{3}, \tfrac{3}{2}, \tfrac{5}{3}}{\frac{16}{27}} = \frac{3}{16} \pi^2,
    \end{equation}
    since $P_3 (\nu ; 1)$ goes to zero as the dimension goes to infinity.
  \end{enumerate}
  The case of $n$-step walks, with $n \geq 4$, is much less accessible
  \cite[\S5]{b3g}, and it would be interesting to obtain a more
  complete extension of Kluyver's result to higher dimensions.
\end{example}

\begin{example}
  \label{eg:p:diminf}Asymptotically,  for $x>0$ as $\nu \rightarrow \infty$,
  \begin{equation*}
    j_{\nu} ( t) \sim \exp \left( - \frac{t^2}{4 \nu + 2} \right) .
  \end{equation*}
  It may thus be derived from the Bessel integral representation
  \eqref{eq:pn:bessel} that asymptotically, as the dimension goes to infinity,
  $p_n ( \nu ; x) \sim q_n ( \nu ; x)$, where
  \begin{equation*}
    q_n ( \nu ; x) = \frac{2^{- \nu}}{\nu !} \left( \frac{2 \nu + 1}{n}
     \right)^{\nu + 1} x^{2 \nu + 1} \exp \left( - \frac{2 \nu + 1}{2 n} x^2
     \right)
  \end{equation*}
  for $x \geq 0$ and $q_n ( \nu ; x) = 0$ for $x < 0$. The probability
  density $q_n ( \nu ; x)$ describes a scaled chi distribution with $2 \nu + 2
  = d$ degrees of freedom. Its average is
  \begin{equation*}
    \int_0^{\infty} x q_n ( \nu ; x) \mathd x = \sqrt{\frac{2 n}{2 \nu + 1}}
     \frac{\Gamma \left( \nu + \tfrac{3}{2} \right)}{\Gamma ( \nu + 1)},
  \end{equation*}
  which converges to $\sqrt{n}$ as $\nu \rightarrow \infty$. More generally,
  for the $s$th moment, with $s > - 1$,
  \begin{eqnarray*}
    \int_0^{\infty} x^s q_n ( \nu ; x) \mathd x & = & \left( \frac{2 n}{2 \nu
    + 1} \right)^{s / 2}  \frac{\Gamma \left( \nu + \tfrac{s}{2} + 1
    \right)}{\Gamma ( \nu + 1)}\\
    & \sim & \left( \frac{2 n \nu}{2 \nu + 1} \right)^{s / 2} \left( 1 +
    \frac{s ( s + 2)}{8 \nu} + O ( \nu^{- 2}) \right),
  \end{eqnarray*}
  as $\nu \rightarrow \infty$, and it is straightforward to compute further
  terms of this asymptotic expansion. The fact that the $s$th moment
  approaches $n^{s / 2}$ for large dimensions, of course, reflects the
  observation from Figures~\ref{fig:p3}, \ref{fig:p4} and \ref{fig:p5} that the
  probability densities $p_n ( \nu ; x)$ center and spike at approximately $\sqrt{n}$.
\end{example}

These Bessel integral representations also allow us to deduce the regularity
of the density functions.

\begin{corollary}\dueto{Regularity of the density}
  \label{cor:pn:reg}The density $p_n (\nu ; x)$ is $m$ times continuously
  differentiable for $x > 0$ if
  \begin{equation*}
    n > \frac{m + 1}{\nu + 1 / 2} + 1, \hspace{1em} \text{or, equivalently,}
     \hspace{1em} m < (n - 1) (\nu + 1 / 2) - 1.
  \end{equation*}
\end{corollary}

\begin{proof}
  It follows from \eqref{eq:pn:bessel}, that
  \begin{equation}
    p_n (\nu ; x) = \frac{2^{- 2 \nu}}{\nu !^2} \int_0^{\infty} (t x)^{2 \nu +
    1} j_{\nu} (t x) j_{\nu}^n (t) \mathd t. \label{eq:pn:besselj}
  \end{equation}
  Observe that this integral converges absolutely if $2 \nu + 1 - (n + 1) (\nu
  + 1 / 2) < - 1$, in which case $p_n (\nu ; x)$ is continuous. Repeatedly
  differentiating under the integral sign as long as is permitted, we conclude
  that $p_n (\nu ; x)$ is $m$ times continuously differentiable if
  \begin{equation*}
    2 \nu + 1 - (n + 1) (\nu + 1 / 2) < - 1 - m,
  \end{equation*}
  and it only remains to solve for $n$, respectively $m$.
\end{proof}

\begin{example}
  In the case $d = 2$, this implies that $p_n (0 ; x)$ is in $C^0$ for $n >
  3$, in $C^1$ for $n > 5$, in $C^2$ for $n > 7$, and so on. Indeed, $p_3 (0 ;
  x)$ has a logarithmic singularity at $x = 1$, and $p_4 (0 ; x)$ as well as
  $p_5 (0 ; x)$ are not differentiable at $x = 2, 4$ and $x = 1, 3, 5$,
  respectively. \end{example}

\begin{corollary}
  Let $n \geq 4$ and $d \geq 2$ such that $(n, d) \neq (4, 2)$.
  Then,
  \begin{equation}
    \frac{1}{(2 \nu + 1) !} p_n^{(2 \nu + 1)} (\nu ; 0) = p_{n - 1} (\nu ; 1),
    \label{eq:pn:diffrel0}
  \end{equation}
  where the derivative is understood to be taken from the right.
\end{corollary}

\begin{proof}
  Starting with \eqref{eq:pn:besselj}, we differentiate $2 \nu + 1$ many
  times, and compare with
  \begin{equation*}
    p_{n - 1} (\nu ; 1) = \frac{2^{- 2 \nu}}{\nu !^2} \int_0^{\infty} t^{2
     \nu + 1} j_{\nu}^n (t) \mathd t.
  \end{equation*}

\end{proof}

In the case $d = 2$, this reduces to $p_n' (0 ; 0) = p_{n - 1} (0 ; 1)$, which
was crucial in \cite{bswz-densities} for explicitly evaluating $p_5' (0 ;
0)$.

\begin{example}
  As long as $p_n (\nu ; x)$ is sufficiently differentiable at $x = 1$, we can
  further relate the value $p_n (\nu ; 1)$, occuring in
  \eqref{eq:pn:diffrel0}, to corresponding values of derivatives. Let us
  illustrate this by showing that
  \begin{equation}
    p_n' (\nu ; 1) = \frac{2 n \nu + n - 1}{n + 1} p_n (\nu ; 1)
    \label{eq:pn:1:d}
  \end{equation}
  for all $n \geq 3$, $\nu > 0$ such that $(n, \nu) \neq (3, 1 / 2)$. In
  that case, we may, as in Corollary~\ref{cor:pn:reg}, differentiate
  \eqref{eq:pn:besselj} under the integral sign to obtain
  \begin{equation*}
    p_n' (\nu ; 1) = (2 \nu + 1) p_n (\nu ; 1) + \frac{2^{- 2 \nu}}{\nu !^2}
     \int_0^{\infty} t^{2 \nu + 2} j_{\nu}' (t) j_{\nu}^n (t) \mathd t
  \end{equation*}
  On the other hand, integrating \eqref{eq:pn:besselj} by parts, we find
  \begin{equation*}
    p_n (\nu ; 1) = - \frac{n + 1}{2 \nu + 2} \frac{2^{- 2 \nu}}{\nu !^2}
     \int_0^{\infty} t^{2 \nu + 2} j_{\nu}^n (t) j_{\nu}' (t) \mathd t.
  \end{equation*}
  Combining these, we arrive at \eqref{eq:pn:1:d}.
\end{example}

The densities of an $n$-step walk can be related to the densities of an $(n -
1)$-step walk by the following generalization of \cite[(2.70)]{hughes-rw}.

\begin{theorem}\dueto{Recursion for the density}
  \label{thm:pn:rec}For $x > 0$ and $n = 1, 2, \ldots$, the function
  \begin{equation*}
    \psi_n (\nu ; x) = \frac{\nu !}{2 \pi^{\nu + 1}}  \frac{p_n (\nu ;
     x)}{x^{2 \nu + 1}}
  \end{equation*}
  satisfies
  \begin{equation}
    \psi_n (\nu ; x) = \frac{\nu !^2 2^{2 \nu}}{(2 \nu) ! \pi} \int_{- 1}^1
    \psi_{n - 1} (\nu ; \sqrt{1 + 2 \lambda x + x^2}) (1 - \lambda^2)^{\nu - 1
    / 2} \mathd \lambda . \label{eq:pn:rec}
  \end{equation}
\end{theorem}

\begin{proof}
  Recall from \eqref{eq:pn:posdis} that the probability density of the
  position after $n$ steps in $\mathbb{R}^d$ is
  \begin{equation*}
    p_n (\nu ; \boldsymbol{x}) = \psi_n (\nu ; \|\boldsymbol{x}\|) .
  \end{equation*}
  Since the steps of our walks are uniformly distributed vectors on the unit
  sphere in $\mathbb{R}^d$, we have
  \begin{equation*}
    p_n (\nu ; \boldsymbol{x}) = \int_{\|\boldsymbol{y}\| = 1} p_{n - 1} (\nu ;
     \boldsymbol{x}-\boldsymbol{y}) \mathd S = \int_{\|\boldsymbol{y}\| = 1} \psi_{n
     - 1} (\nu ; \|\boldsymbol{x}-\boldsymbol{y}\|) \mathd S,
  \end{equation*}
  where $\mathd S$ denotes the normalized surface measure of the unit sphere.
  After introducing $d$-dimensional spherical polar coordinates, as detailed,
  for instance, in \cite[p.~61]{hughes-rw}, and a straightforward change of
  variables we arrive at \eqref{eq:pn:rec}.
\end{proof}

Finally, a computationally more accessible expression for the densities $p_n
(\nu ; x)$ is given by the following generalization of a formula, which was
derived by Broadhurst \cite{broadhurst-rw} in the case of two dimensions,
that is, $\nu = 0$. Note that \eqref{eq:pn:bessel} is the special case $k = 0$
in \eqref{eq:pn:besselD}.

\begin{theorem} {\dueto{Bessel integral for the densities, II}}
\label{thm:pn:besselD}Let $n
  \geq 2$ and $d \geq 2$. For any nonnegative integer $k$, and $x >
  0$,
  \begin{equation}
    p_n (\nu ; x) = \frac{2^{- \nu}}{\nu !}  \frac{1}{x^{2 k}} \int_0^{\infty}
    (t x)^{\nu + k + 1} J_{\nu + k} (t x) \left( - \frac{1}{t}
    \frac{\mathd}{\mathd t} \right)^k j_{\nu}^n (t) \mathd t.
    \label{eq:pn:besselD}
  \end{equation}
\end{theorem}

\begin{proof}
  As in \cite{broadhurst-rw}, we proceed by induction on $k$. The case $k =
  0$ is \eqref{eq:pn:bessel}. Suppose that \eqref{eq:pn:besselD} is known for
  some $k$. By the Bessel function identity
  \begin{equation}
    \frac{\mathd}{\mathd z} ( z^{\nu} J_{\nu} (z)) = z^{\nu} J_{\nu - 1} (z)
    \label{eq:J:D}
  \end{equation}
  we find that, for any smooth function $g (t)$ which is sufficiently small at
  $0$ and $\infty$,
  \begin{eqnarray*}
    \int_0^{\infty} (t x)^{\alpha + 1} J_{\alpha} (t x) g (t) \mathd t & = &
    \frac{1}{x} \int_0^{\infty} \left[ \frac{\mathd}{\mathd t} (t x)^{\alpha +
    1} J_{\alpha} (t x) \right] g (t) \mathd t\\
    & = & - \frac{1}{x} \int_0^{\infty} (t x)^{\alpha + 1} J_{\alpha} (t x)
    \left[ \frac{\mathd}{\mathd t} g (t) \right] \mathd t\\
    & = & \frac{1}{x^2} \int_0^{\infty} (t x)^{\alpha + 2} J_{\alpha} (t x)
    \left[ - \frac{1}{t} \frac{\mathd}{\mathd t} g (t) \right] \mathd t.
  \end{eqnarray*}
  In the second step, we used integration by parts and assumed that $g (t)$ is
  such that
  \begin{equation}
    (t x)^{\alpha + 1} J_{\alpha} (t x) g (t) \label{eq:Jg}
  \end{equation}
  vanishes as $t \rightarrow 0$ and $t \rightarrow \infty$. In the present
  case, $\alpha = \nu + k$ and $g (t) = \left( - \frac{1}{t}
  \frac{\mathd}{\mathd t} \right)^k j_{\nu}^n (t)$. In order to complete the
  proof of \eqref{eq:pn:besselD}, it only remains to demonstrate that
  \eqref{eq:Jg} indeed vanishes as required.

  As $t \rightarrow 0$, we have $(t x)^{\alpha + 1} J_{\alpha} (t x) = O (t^{2
  \alpha + 1}) = O (t^{2 \nu + 2 k + 1})$ while $g (t) = O (1)$, because
  $j_{\nu}^n (t)$ is an even function. Hence, for $\nu > - 1$, the term
  \eqref{eq:Jg} indeed vanishes as $t \rightarrow 0$. On the other hand, as $t
  \rightarrow \infty$, we have $(t x)^{\alpha + 1} J_{\alpha} (t x) = O
  (t^{\alpha + 1 / 2}) = O (t^{\nu + k + 1 / 2})$ by \eqref{eq:j:infty}.
  Moreover, \eqref{eq:j:infty} and \eqref{eq:J:D} imply that $j_{\nu}^{(m)}
  (t) = O (t^{- \nu - 1 / 2})$, as $t \rightarrow \infty$, and therefore that
  $g (t) = O (t^{- n (\nu + 1 / 2) - k})$. Since $n > 1$, it follows that
  \eqref{eq:Jg} also vanishes as $t \rightarrow \infty$.
\end{proof}

We note that the $\left( - \frac{1}{t} \frac{\mathd}{\mathd t} \right)^k
j_{\nu}^n (t)$ in the integrand of \eqref{eq:pn:besselD} may be expressed as a
finite sum of products of Bessel functions. This is made explicit in
Remark~\ref{rk:j:D:k}.

\subsection{The probability densities in odd dimensions}\label{ssec:odd}

In the case of odd dimension $d$, the Bessel functions in
Section~\ref{sec:basic} have half-integer index $\nu$ and are therefore
elementary, so that the situation is fairly well understood since Rayleigh
\cite{rayleigh-flights,hughes-rw}.
In particular, the probability density functions $p_n (\nu
; x)$ are piecewise polynomial in odd dimensions. This is made explicit by the
following theorem, which is obtained in \cite{garcia-odd}, translated to our
notation.

\begin{theorem}
  {\dueto{Density in odd dimensions \cite[Theorem~2.6]{garcia-odd}}}\label{thm:pn:odd}
  Assume that the dimension $d=2m+1$ is an odd number. Then, for $0 < x < n$,
  \begin{equation*}
    p_n (m - \tfrac{1}{2} ; x) =
    \frac{(2 x)^{2 m} \Gamma (m)}{\Gamma (2 m)}
     \left( - \frac{1}{2 x}  \frac{\mathd}{\mathd x} \right)^m P_{m, n} (x),
  \end{equation*}
  where $P_{m, n} (x)$ is the piecewise polynomial obtained from convoluting
  \begin{equation*}
    \frac{\Gamma (m + 1 / 2)}{\Gamma (1 / 2) \Gamma (m)} \left\{
     \begin{array}{ll}
       (1 - x^2)^{m - 1}, & \text{for $x \in [- 1, 1]$,}\\
       0, & \text{otherwise},
     \end{array} \right.
  \end{equation*}
  $n-1$ times with itself.
\end{theorem}

\begin{example}
  In the case $n = 3$ and $d = 3$, we have $m = 1$ and
  \begin{equation*}
    P_{1, 1} (x) = \frac{1}{2} \left\{ \begin{array}{ll}
       1, & \text{if $|x| \leq 1$,}\\
       0, & \text{otherwise},
     \end{array} \right.
  \end{equation*}
  as well as
  \begin{equation*}
    P_{1, 2} (x) = \int_{- \infty}^{\infty} P_{1, 1} (t) P_{1, 1} (x - t)
     \mathd t = \frac{1}{4} \left\{ \begin{array}{ll}
       2 - |x|, & \text{if $|x| \leq 2$,}\\
       0, & \text{otherwise},
     \end{array} \right.
  \end{equation*}
  and, hence,
  \begin{equation*}
    P_{1, 3} (x) = \int_{- \infty}^{\infty} P_{1, 2} (t) P_{1, 1} (x - t)
     \mathd t = \frac{1}{8} \left\{ \begin{array}{ll}
       3 - x^2, & \text{if $|x| \leq 1$,}\\
       \frac{1}{2} (|x| - 3)^2, & \text{if $1 < |x| \leq 3$,}\\
       0, & \text{otherwise} .
     \end{array} \right.
  \end{equation*}
  We thus find that, for $0 \leq x \leq 3$,
  \begin{equation*}
    p_3 ( \tfrac{1}{2} ; x) = 4 x^2 \left( - \frac{1}{2 x}
     \frac{\mathd}{\mathd x} \right) P_{1, 3} (x) = \frac{x}{4} \left\{
     \begin{array}{ll}
       2 x, & \text{if $0 \leq x \leq 1$,}\\
       3 - x, & \text{if $1 < x \leq 3$.}
     \end{array} \right.
  \end{equation*}
  Since $j_{1/2}(x)= \sinc(x)=\sin(x)/x$, evaluation of the densities in three dimensions can also be approached using the tools provided by \cite{dbjb}.

  Similarly, we obtain, for instance,
  \begin{equation}\label{eq:p4:3d}
    p_4 ( \tfrac{1}{2} ; x) = \frac{x}{16} \left\{ \begin{array}{ll}
       x (8 - 3 x), & \text{if $0 \leq x \leq 2$,}\\
       (4 - x)^2, & \text{if $2 < x \leq 4$.}
     \end{array} \right.
  \end{equation}
  We note that Theorem \ref{thm:pn:odd} can be usefully implemented in a computer algebra system such as \emph{Maple} or \emph{Mathematica}.
\end{example}

\begin{example}\dueto{Moments in odd dimensions}
  By integrating \eqref{eq:p4:3d}, we are able to symbolically compute the
  corresponding moments, as introduced in \eqref{eq:Wn}, as
  \[ W_4(1/2;s)=2^{s+3}\,{\frac {1-{2}^{s+2}}{ \left( s+2 \right) }\frac{1}{ \left( s+4 \right)  \left( s+3 \right) }}, \]
  which has a removable singularity $-2$ and poles at $-3$ and $-4$.
  Likewise, in five dimensions,
  \[W_4(3/2;s)=
    \frac{(12)^3{2}^{s+1} \left( {s}^{3}+27\,{s}^{2}+230\,s+616+64\,{2}^{s
    } \left( {s}^{3}+15\,{s}^{2}+62\,s+56 \right)  \right) }{ \left( s+12
      \right)  \left( s+10 \right)  \left( s+9 \right)  \left( s+8 \right)
      \left( s+7 \right)  \left( s+6 \right)  \left( s+5 \right)  \left( s+
    4 \right)  \left( s+2 \right) }
  \]
  with poles at $-5,-7,-8,-9,-10,-12 $ and removable singularities at the
  other apparent poles. Thus, even this elementary evaluation has subtle structure.
\end{example}

\subsection{The moment functions}\label{ssec:mom}

Theorem~\ref{thm:pn:besselD} allows us to prove a corresponding Bessel
integral representation for the \emph{moment function}
\begin{equation*}
  W_n (\nu ; s) = \int_0^{\infty} x^s p_n (\nu ; x) \mathd x
\end{equation*}
of the distance to the origin after $n$ random steps. The following result
generalizes \cite[Theorem~1]{broadhurst-rw} from two to arbitrary
dimensions.

\begin{theorem}
  {\dueto{Bessel integral for the moments}}\label{thm:W:bessel}Let $n
  \geq 2$ and $d \geq 2$. For any nonnegative integer $k$,
  \begin{equation}
    W_n ( \nu ; s) = \frac{2^{s - k + 1} \Gamma \left( \frac{s}{2} + \nu + 1
    \right)}{\Gamma (\nu + 1) \Gamma \left( k - \frac{s}{2} \right)}
    \int_0^{\infty} x^{2 k - s - 1} \left( - \frac{1}{x}  \frac{\mathd}{\mathd
    x} \right)^k j_{\nu}^n ( x) \mathd x, \label{eq:W:bessel}
  \end{equation}
  provided that $k - n (\nu + 1 / 2) < s < 2 k$.
\end{theorem}

\begin{proof}
  Using Theorem~\ref{thm:pn:besselD}, we have
  \begin{equation*}
    W_n ( \nu ; s) = \frac{2^{- \nu}}{\nu !} \int_0^{\infty} x^{s - 2 k}
     \int_0^{\infty} (t x)^{\nu + k + 1} J_{\nu + k} (t x) \left( -
     \frac{1}{t} \frac{\mathd}{\mathd t} \right)^k j_{\nu}^n (t) \mathd t
     \mathd x.
  \end{equation*}
  Interchanging the order of integration and substituting $z = t x$, we obtain
  \begin{equation*}
    W_n ( \nu ; s) = \frac{2^{- \nu}}{\nu !} \int_0^{\infty} t^{2 k - s - 1}
     \left[ \left( - \frac{1}{t} \frac{\mathd}{\mathd t} \right)^k j_{\nu}^n
     (t) \right] \int_0^{\infty} z^{\nu + s - k + 1} J_{\nu + k} (z) \mathd z
     \mathd t.
  \end{equation*}
  The inner integral may be evaluated using the standard Bessel integral evaluation
  \begin{equation*}
    \int_0^{\infty} z^a J_{\nu} ( z) \mathd z = \frac{2^a \Gamma \left(
     \frac{1 + a + \nu}{2} \right)}{\Gamma \left( \frac{1 - a + \nu}{2}
     \right)},
  \end{equation*}
  which holds for $a$ and $\nu$ such that $a + \nu > - 1$ and $a < 1 / 2$. We
  conclude that
  \begin{eqnarray*}
    W_n ( \nu ; s) & = & \frac{2^{- \nu}}{\nu !} \int_0^{\infty} t^{2 k - s -
    1} \left[ \left( - \frac{1}{t} \frac{\mathd}{\mathd t} \right)^k j_{\nu}^n
    (t) \right] \frac{2^{\nu + s - k + 1} \Gamma \left( \frac{s}{2} + \nu + 1
    \right)}{\Gamma \left( k - \frac{s}{2} \right)} \mathd t,
  \end{eqnarray*}
  which is the desired result. For the evaluation of the intermediate Bessel
  integral, we assumed $s > - 2 \nu - 2 = - d$ and $s < k - \nu - 1 / 2$, and
  so \eqref{eq:W:bessel} holds for all $s$ in this non-empty strip provided
  that the original integral converges. Using the asymptotic bounds from the
  proof of Theorem~\ref{thm:pn:besselD}, we note that the integral
  \eqref{eq:W:bessel} converges absolutely in the strip $k - n (\nu + 1 / 2) <
  s < 2 k$. Analytic continuation therefore implies that \eqref{eq:W:bessel}
  holds for all $s$ in this strip.
\end{proof}

We deduce the following from \eqref{eq:W:bessel}, with $k = 0$. This extends
\cite[Proposition 2.4]{bsw-rw2}. In particular, we observe that, for $n >
2$, the first pole of $W_n (\nu ; s)$ occurs at $s = - ( 2 \nu + 2) = - d$.

\begin{corollary}\dueto{Poles and residues of the moments}
  \label{cor:W:res}Let $n > 2$. In the half-plane $\Re s > - n (\nu + 1
  / 2)$, the moment functions $W_n (\nu ; s)$ are analytic apart from simple
  poles at $s = - d - 2 m$ for integers $m$ such that $0 \leq m <
  \frac{n}{2} \left( \frac{d}{2} - \frac{1}{2} \right) - \frac{d}{2}$. The
  residues of these poles are
  \begin{equation*}
    \operatorname{Res}_{s = - d - 2 m} W_n (\nu ; s) = \frac{2^{- 2 \nu - 2 m}}{\nu !
     ( \nu + m) !} \frac{( - 1)^m}{m!} \int_0^{\infty} x^{2 \nu + 2 m + 1}
     j_{\nu}^n ( x) \mathd x.
  \end{equation*}
\end{corollary}

\begin{proof}
  By equation \eqref{eq:W:bessel}, with $k = 0$, we have
  \begin{equation*}
    W_n ( \nu ; s) = \frac{2^{s + 1} \Gamma \left( \frac{s}{2} + \nu + 1
     \right)}{\Gamma (\nu + 1) \Gamma \left( - \frac{s}{2} \right)}
     \int_0^{\infty} x^{- s - 1} j_{\nu}^n ( x) \mathd x,
  \end{equation*}
  valid for $s$ in the strip $- n (\nu + 1 / 2) < s < 0$, in which the
  integral converges absolutely. In the region of interest, the only poles are
  contributed by the factor $\Gamma \left( \frac{s}{2} + \nu + 1 \right)$,
  which, as a function in $s$, has simple poles at $s = - d - 2 m$, for $m =
  0, 1, 2, \ldots$, with residue $2 \frac{( - 1)^m}{m!}$.
\end{proof}

Note that the value for the residue of $W_n ( \nu ; s)$ at $s = - d$ agrees
with
\begin{equation*}
  p_{n - 1} (\nu ; 1) = \frac{1}{(2 \nu + 1) !} p_n^{(2 \nu + 1)} (\nu ; 0)
\end{equation*}
from \eqref{eq:pn:diffrel0} provided that these values are finite.

\begin{example}
  \label{eg:W3:poles}The moment functions $W_3 ( \nu ; s)$ are plotted in
  Figure~\ref{fig:W3} for $\nu = 0, 1, 2$. The first pole of $W_3 ( \nu ; s)$
  occurs at $s = - d$, is simple and has residue
  \begin{equation*}
    \operatorname{Res}_{s = - d} W_3 ( \nu ; s) = \frac{2^{- 2 \nu}}{\nu !^2}
     \int_0^{\infty} x^{2 \nu + 1} j_{\nu}^3 ( x) \mathd x = \frac{2}{\sqrt{3}
     \pi}  \frac{3^{\nu}}{\binom{2 \nu}{\nu}},
  \end{equation*}
  as follows from the Bessel integral formula \eqref{eq:int:J3}. In order to
  record some more general properties of the pole structure of $W_3 ( \nu ;
  s)$, we use the fact that $W_3 ( \nu ; s)$ satisfies a functional equation,
  \eqref{eq:rec3}, which relates the three terms $W_3 ( \nu ; s)$, $W_3 ( \nu
  ; s + 2)$ and $W_3 ( \nu ; s + 4)$. By reversing this functional equation,
  we find that the residues $r_k = \operatorname{Res}_{s = - d - 2 k} W_3 ( \nu ; s)$
  satisfy the recursion
  \begin{eqnarray}
    &  & 9 ( k + 1) ( k + \nu + 1) r_{k + 1} \nonumber\\
    & = & \tfrac{1}{2} \left( 20 \left( k + \tfrac{1}{2} \right)^2 - 20
    \left( k + \tfrac{1}{2} \right) \nu - 4 \nu^2 + 1 \right) r_k - ( k - \nu)
    ( k - 2 \nu) r_{k - 1},  \label{eq:rec3r}
  \end{eqnarray}
  with $r_{- 1} = 0$ and $r_0 = \frac{2}{\sqrt{3} \pi}
  \frac{3^{\nu}}{\binom{2 \nu}{\nu}}$. Observe that, when $\nu = 0$, the
  recursion for these residues is essentially the same as the recurrence
  \eqref{eq:rec3} for the corresponding even moments (with $u_k$ replaced by
  $3^{2 k} r_k$). As recorded in \cite[Proposition 2.4]{bsw-rw2}, this lead
  to
  \begin{equation*}
    \operatorname{Res}_{s = - 2 ( k + 1)} W_3 ( 0 ; s) = \frac{2}{\sqrt{3} \pi}
     \frac{W_3 ( 0 ; 2 k)}{3^{2 k}} .
  \end{equation*}
  Define, likewise, the numbers $V_3 ( \nu ; k)$ by
  \begin{equation*}
    \operatorname{Res}_{s = - 2 ( \nu + k + 1)} W_3 ( \nu ; s) = \frac{2}{\sqrt{3}
     \pi}  \frac{3^{\nu}}{\binom{2 \nu}{\nu}} \frac{V_3 ( \nu ; k)}{3^{2 k}} .
  \end{equation*}
  In analogy with \eqref{eq:rec3r}, we find that $u_k = V_3 ( \nu ; k)$ solves
  the three-term recurrence
  \begin{eqnarray*}
    &  & ( k + 1) ( k + \nu + 1) u_{k + 1}\\
    & = & \tfrac{1}{2} \left( 20 \left( k + \tfrac{1}{2} \right)^2 - 20
    \left( k + \tfrac{1}{2} \right) \nu - 4 \nu^2 + 1 \right) u_k - 9 ( k -
    \nu) ( k - 2 \nu) u_{k - 1},
  \end{eqnarray*}
  with $u_{- 1} = 0$ and $u_0 = 1$. For small dimensions, initial values for
  $V_3 ( \nu ; k)$ are given by
  \begin{eqnarray*}
    d = 2 \hspace{1em} ( \nu = 0) & : & 1, 3, 15, 93, 639, 4653, 35169,
    272835, 2157759, \ldots\\
    d = 4 \hspace{1em} ( \nu = 1) & : & 1, - 2, - 2, - 6, - 24, - 114, - 606,
    - 3486, - 21258, \ldots\\
    d = 6 \hspace{1em} ( \nu = 2) & : & 1, - 5, 6, 2, 6, 18, 66, 278, 1296,
    \ldots\\
    d = 8 \hspace{1em} ( \nu = 3) & : & 1, - \tfrac{15}{2}, 21, - 20, 0, - 9,
    - 20, - 60, - 210, \ldots
  \end{eqnarray*}
  Note the increasingly irregular behaviour of $V_3 ( \nu ; k)$ as $d$
  increases. In particular, in dimension $8$, we find $V_3 ( 3 ; 4) = 0$,
  which signifies the disappearance of the perhaps expected pole of $W_3 ( 3 ; s)$ at
  $s = - 16$.
\end{example}

\begin{figure}[h]
  \centering
  \includegraphics[width=4.5cm]{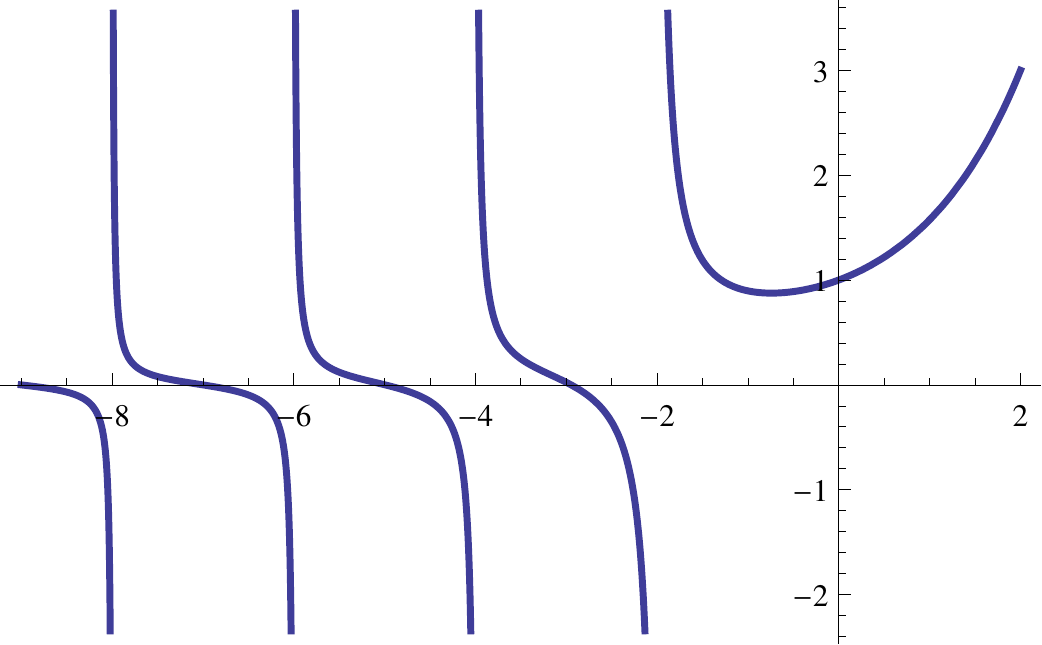}
  \quad
  \includegraphics[width=4.5cm]{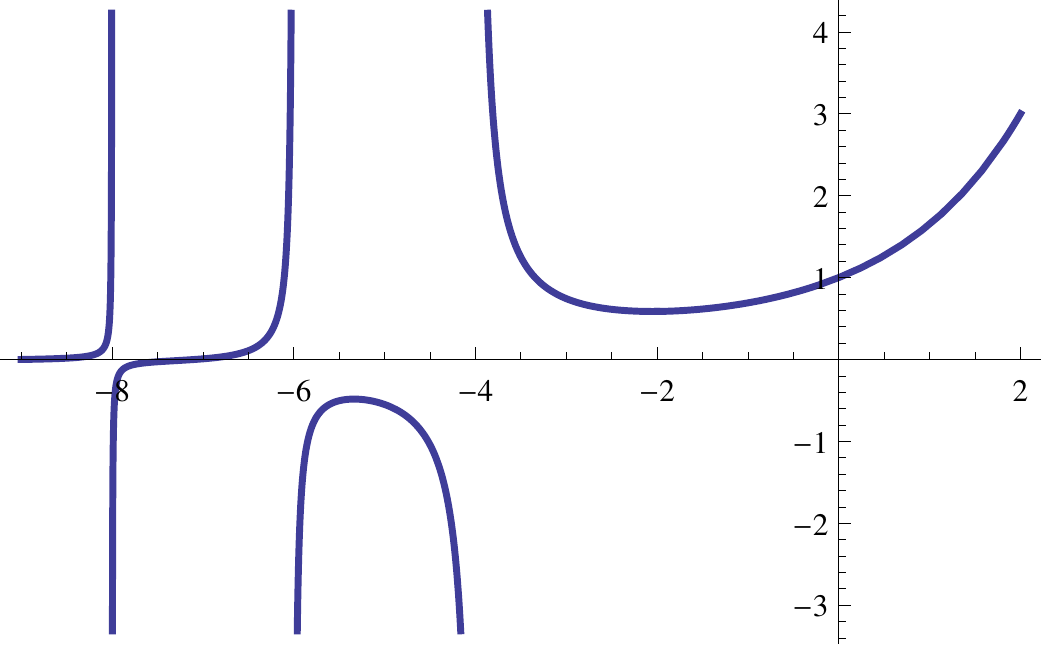}
  \quad
  \includegraphics[width=4.5cm]{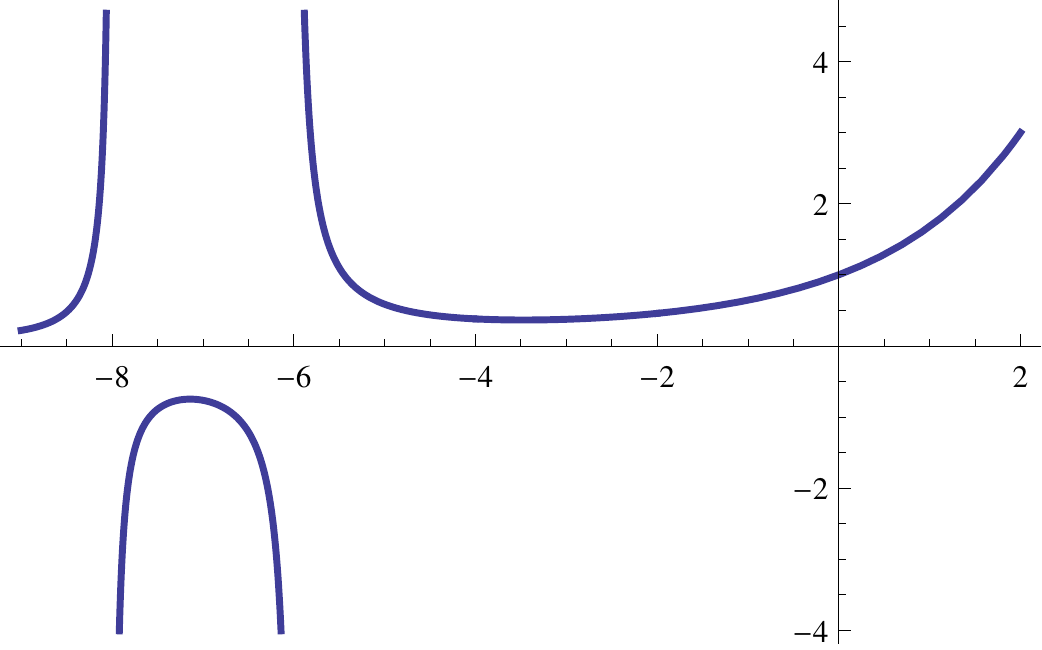}
  \caption{\label{fig:W3}$W_3 ( \nu ; s)$ on $[ - 9, 2]$ for $\nu = 0, 1, 2$.}
\end{figure}

\begin{example}
  \label{eg:W4:poles}In the case $n = 4$ and $\nu = 1$, that is $d = 4$, the
  moment function $W_4 ( 1 ; s)$ has a simple pole at $s = - 4$ with residue
  \begin{equation*}
    \operatorname{Res}_{s = - 4} W_4 (1 ; s) = \frac{1}{4} \int_0^{\infty} x^3 j_1^4
     ( x) \mathd x = \frac{4}{\pi^2},
  \end{equation*}
  and is otherwise analytic in the half-plane $\Re s > - 6$. At $s = -
  6$, on the other hand, $W_4 ( 1 ; s)$ has a double pole. Indeed, analyzing
  the functional equations that arise from \eqref{eq:rec4} and
  \eqref{eq:W4:rec:nu:s}, we derive that
  \begin{equation*}
    \lim_{s \rightarrow - 6} ( s + 6)^2 W_4 ( 1 ; s) = - \frac{1}{2} \lim_{s
     \rightarrow - 2} ( s + 2)^2 W_4 ( 0 ; s) = - \frac{3}{4 \pi^2},
  \end{equation*}
  where in the last equality we used the known value from the planar case
  \cite[Example 4.3]{bswz-densities}. Similarly, we obtain
  \begin{eqnarray*}
    \operatorname{Res}_{s = - 6} W_4 ( 1 ; s) & = & - \frac{1}{2} \operatorname{Res}_{s = - 2}
    W_4 ( 0 ; s) + \frac{1}{24} \lim_{s \rightarrow - 2} ( s + 2)^2 W_4 ( 0 ;
    s)\\
    & = & \frac{1}{16 \pi^2} - \frac{9 \log ( 2)}{4 \pi^2} .
  \end{eqnarray*}
  In the higher-dimensional case, $W_4 ( \nu ; s)$ has poles at $s = - d - 2
  m$ for $m = 0, 1, 2, \ldots$, which are initially simple but turn into poles
  of order (up to) $2$ beginning at $s = - ( 4 \nu + 2)$.
\end{example}

The approach indicated in Examples~\ref{eg:W3:poles} and \ref{eg:W4:poles}
enables us to determine, at least in principle, the pole structure of the
moment functions $W_n ( \nu ; s)$ in each case. We do not pursue such a more
detailed analysis herein.

We next follow the approach of \cite{broadhurst-rw} to obtain a summatory
expression for the even moments from Theorem~\ref{thm:pn:besselD}.

\begin{theorem}
  \label{thm:W:even}{\dueto{Multinomial sum for the moments}}The even moments
  of an $n$-step random walk in dimension $d$ are given by
  \begin{equation*}
    W_n ( \nu ; 2 k) = \frac{( k + \nu) ! \nu !^{n - 1}}{( k + n \nu) !}
     \sum_{k_1 + \cdots + k_n = k} \binom{k}{k_1, \ldots, k_n} \binom{k + n
     \nu}{k_1 + \nu, \ldots, k_n + \nu} .
  \end{equation*}
\end{theorem}

\begin{proof}
  Replacing $k$ by $k + 1$ in \eqref{eq:W:bessel} and setting $s = 2 k$, we
  obtain
  \begin{eqnarray}
    W_n ( \nu ; 2 k) & = & \frac{2^k ( k + \nu) !}{\nu !} \int_0^{\infty} -
    \frac{\mathd}{\mathd x} \left( - \frac{1}{x}  \frac{\mathd}{\mathd x}
    \right)^k j_{\nu}^n ( x) \mathd x \nonumber\\
    & = & \left[ \frac{( k + \nu) !}{\nu !} \left( - \frac{2}{x}
    \frac{\mathd}{\mathd x} \right)^k j_{\nu}^n ( x) \right]_{x = 0} .
    \label{eq:W:even:jn}
  \end{eqnarray}
  Observe that, at the level of formal power series, we have
  \begin{equation*}
    \left[ \left( - \frac{2}{x}  \frac{\mathd}{\mathd x} \right)^k \sum_{m
     \geq 0} a_m \left( - \frac{x^2}{4} \right)^m \right]_{x = 0} = k!a_k
     .
  \end{equation*}
  Recall from \eqref{eq:j} the series
  \begin{equation*}
    j_{\nu} ( x) = \nu ! \sum_{m \geq 0} \frac{( - x^2 / 4)^m}{m! ( m +
     \nu) !},
  \end{equation*}
  to conclude that
  \begin{equation*}
    W_n ( \nu ; 2 k) = \frac{( k + \nu) !}{\nu !} \nu !^n k! \sum_{m_1 +
     \cdots + m_n = k} \frac{1}{m_1 ! \cdots m_n !}  \frac{1}{( m_1 + \nu) !
     \cdots ( m_n + \nu) !},
  \end{equation*}
  which is equivalent to the claimed formula.
\end{proof}

\begin{remark}
  \label{rk:j:D:k}Proceeding as in the proof of Theorem~\ref{thm:W:even}, we
  observe that
  \begin{equation*}
    \left( - \frac{2}{x}  \frac{\mathd}{\mathd x} \right)^k j_{\nu} (x) =
     \frac{\nu !}{(\nu + k) !} j_{\nu + k} (x),
  \end{equation*}
  and hence
  \begin{eqnarray*}
    &  & \left( - \frac{2}{x}  \frac{\mathd}{\mathd x} \right)^k j_{\nu_1}
    (x) \cdots j_{\nu_n} (x)\\
    & = & \sum_{k_1 + \cdots + k_n = k} \frac{k!}{k_1 ! \cdots k_n !}
    \frac{\nu_1 ! \cdots \nu_n !}{(k_1 + \nu_1) ! \cdots (k_n + \nu_n) !}
    j_{\nu_1 + k_1} (x) \cdots j_{\nu_n + k_n} (x) .
  \end{eqnarray*}
  If applied to \eqref{eq:W:even:jn}, this (finite) expansion, together with
  $j_{\nu} (0) = 1$, makes the conclusion of Theorem~\ref{thm:W:even}
  apparent. In conjunction with the asymptotics of $j_{\nu}$, we conclude, as
  in the proof of Theorem~\ref{thm:pn:besselD}, that the integrand in
  \eqref{eq:pn:besselD} is $O (t^{- ( n - 1) (\nu + 1 / 2)})$ as $t
  \rightarrow \infty$, so that each additional derivative improves the order
  at $+ \infty$ by $1$ --- at the expense of increasing the size of the
  coefficients. We note, inter alia, that \eqref{eq:W:even:jn} may,
  alternatively, be expressed as
  \begin{equation*}
    j_{\nu}^n ( x) = \nu ! \hspace{0.25em} \sum_{k \geq 0} \frac{W_n
     (\nu ; 2 k)}{k! (k + \nu) !} \left( - \frac{x^2}{4} \right)^k ,
  \end{equation*} which yields a fine alternative generating function for the even moments.
\end{remark}

\begin{example}
  \label{eg:W:smallk}In the case $k = 1$, Theorem~\ref{thm:W:even} immediately
  implies that the second moment of an $n$-step random walk in any dimension
  is
  \begin{equation*}
    W_n (\nu ; 2) = n.
  \end{equation*}
  This was proved in \cite[Theorem~4.2]{wan-phd} using a multi-dimensional
  integral representation and hyper-spherical coordinates. Similarly, we find
  that
  \begin{equation}\label{eq:W:k4}
    W_n (\nu ; 4) =
    \frac{n (n (\nu + 2) - 1)}{\nu + 1} .
  \end{equation}
  More generally, Theorem~\ref{thm:W:even} shows that $W_n (\nu ; 2 k)$ is a
  polynomial of degree $k$ in $n$, with coefficients that are rational
  functions in $\nu$. For instance,
  \begin{equation}
    W_n (\nu ; 6) =
    \frac{n (n^2 (\nu+2)(\nu+3) - 3 n(\nu+3) + 4)}{(\nu+1)^2}
  \end{equation}
  and so on.
\end{example}

Using the explicit expression of the even moments of an $n$-step random walk
in dimension $d$, we derive the following convolution relation.

\begin{corollary}
  \label{cor:W:iter}{\dueto{Moment recursion}} For positive integers $n_1, n_2$, half-integer $\nu$ and nonnegative integer $k$  we have
  \begin{equation}
    W_{n_1 + n_2} ( \nu ; 2 k) = \sum_{j = 0}^k \binom{k}{j} \frac{( k + \nu)
    ! \nu !}{( k - j + \nu) ! ( j + \nu) !} W_{n_1} ( \nu ; 2 j) W_{n_2} ( \nu
    ; 2 ( k - j)) . \label{eq:W:crec}
  \end{equation}
\end{corollary}

Note the special case $n_2 = 1$, that is
\begin{equation}
  W_n ( \nu ; 2 k) = \sum_{j = 0}^k \binom{k}{j} \frac{( k + \nu) ! \nu !}{( k
  - j + \nu) ! ( j + \nu) !} W_{n - 1} ( \nu ; 2 j), \label{eq:W:crec1}
\end{equation}
which allows us to relate the moments of an $n$-step walk to the moments of an
$(n - 1)$-step walk.

\begin{example}\dueto{Integrality of two and four dimensional even moments}
  \label{eg:W:values4d}Corollary~\ref{cor:W:iter} provides an efficient way to
  compute even moments of random walks in any dimension. For illustration, and
  because they are integral, we record the moments of an $n$-step walk in two
  and four dimensions for $n = 2, 3, \ldots, 6$.
  \begin{eqnarray*}
    W_2 ( 0 ; 2 k) & : & 1, 2, 6, 20, 70, 252, 924, 3432, 12870, \ldots\\
    W_3 ( 0 ; 2 k) & : & 1, 3, 15, 93, 639, 4653, 35169, 272835, 2157759,
    \ldots\\
    W_4 ( 0 ; 2 k) & : & 1, 4, 28, 256, 2716, 31504, 387136, 4951552,
    65218204, \ldots\\
    W_5 ( 0 ; 2 k) & : & 1, 5, 45, 545, 7885, 127905, 2241225, 41467725,
    798562125, \ldots\\
    W_6 ( 0 ; 2 k) & : & 1, 6, 66, 996, 18306, 384156, 8848236, 218040696,
    5651108226, \ldots
  \end{eqnarray*}
  For $n = 2$, these are central binomial coefficients, see \eqref{eq:W2},
  while, for $n = 3, 4$, these are Ap\'ery-like sequences, see
  \eqref{eq:W3:d2} and \eqref{eq:W4:d2}. Likewise, the initial even moments for
  four dimensions are as follows.
  \begin{eqnarray*}
    W_2 ( 1 ; 2 k) & : & 1, 2, 5, 14, 42, 132, 429, 1430, 4862, \ldots\\
    W_3 ( 1 ; 2 k) & : & 1, 3, 12, 57, 303, 1743, 10629, 67791, 448023,
    \ldots\\
    W_4 ( 1 ; 2 k) & : & 1, 4, 22, 148, 1144, 9784, 90346, 885868, 9115276,
    \ldots\\
    W_5 ( 1 ; 2 k) & : & 1, 5, 35, 305, 3105, 35505, 444225, 5970725,
    85068365, \ldots\\
    W_6 ( 1 ; 2 k) & : & 1, 6, 51, 546, 6906, 99156, 1573011, 27045906,
    496875786, \ldots
  \end{eqnarray*}
  Observe that the first terms are as determined in Example~\ref{eg:W:smallk}.
  In the two-step case in four dimensions, we find that the even moments are
  the \emph{Catalan numbers} $C_k$, that is
  \begin{equation}
    W_2 ( 1 ; 2 k) = \frac{( 2 k + 2) !}{( k + 1) ! ( k + 2) !} = C_{k + 1}
    , \hspace{1em} C_k \assign \frac{1}{k + 1} \binom{2 k}{k} .
    \label{eq:W2:d4}
  \end{equation}
  This adds another interpretation to the impressive array of quantities that
  are given by the Catalan numbers.

  It is a special property of the random walks in two and four dimensions that
  all even moments are positive integers (compare, for instance,
  \eqref{eq:W:k4}). This is obvious for two dimensions from
  Theorem~\ref{thm:W:even} which, in fact, demonstrates that the moments
  \begin{equation*}
    W_n ( 0 ; 2 k) = \sum_{k_1 + \cdots + k_n = k} \binom{k}{k_1, \ldots,
     k_n}^2
  \end{equation*}
  count {\emph{abelian squares}} \cite{richmond-absq09}. On the other hand,
  to show that the four-dimensional moments $W_n ( 1 ; 2 k)$ are always
  integral, it suffices to recursively apply \eqref{eq:W:crec1} and to note
  that the factors
  \begin{equation}
    \binom{k}{j} \frac{( k + 1) !}{( k - j + 1) ! ( j + 1) !} = \frac{1}{j +
    1} \binom{k}{j} \binom{k + 1}{j} \label{eq:narayana}
  \end{equation}
  are integers for all nonnegative integers $j$ and $k$. The numbers
  \eqref{eq:narayana} are known as {\emph{Narayana numbers}} and occur in
  various counting problems; see, for instance,
  \cite[Problem~6.36]{stanley-ec2}.
\end{example}

\begin{example}\dueto{Narayana numbers}\label{eg:W:matpow}
  The recursion \eqref{eq:W:crec1} for the moments $W_n (\nu ; 2 k)$ is
  equivalent to the following interpretation of the moments as row sums of the
  $n$th power of certain triangular matrices. Indeed, for given $\nu$, let $A
  (\nu)$ be the infinite lower triangular matrix with entries
  \begin{equation}\label{eq:A:kj}
    A_{k, j} (\nu) = \binom{k}{j} \frac{( k + \nu) ! \nu !}{( k - j + \nu) !
     ( j + \nu) !}
  \end{equation}
  for row indices $k = 0, 1, 2, \ldots$ and column indices $j = 0, 1, 2, \ldots$. Then the row sums of $A
  (\nu)^n$ are given by the moments $W_{n + 1} (\nu ; 2 k)$, $k = 0, 1, 2,
  \ldots$. For instance, in the case $\nu = 1$,
  \begin{equation*}
    A (1) = \begin{bmatrix}
       1 & 0 & 0 & 0 & \cdots\\
       1 & 1 & 0 & 0 & \\
       1 & 3 & 1 & 0 & \\
       1 & 6 & 6 & 1 & \\
       \vdots &  &  &  & \ddots
     \end{bmatrix}, \qquad A (1)^3 = \begin{bmatrix}
       1 & 0 & 0 & 0 & \cdots\\
       3 & 1 & 0 & 0 & \\
       12 & 9 & 1 & 0 & \\
       57 & 72 & 18 & 1 & \\
       \vdots &  &  &  & \ddots
     \end{bmatrix},
  \end{equation*}
  with the row sums $1, 2, 5, 14, \ldots$ and $1, 4, 22, 148, \ldots$
  corresponding to the moments $W_2 ( 1 ; 2 k)$ and $W_4 ( 1 ; 2 k)$ as given
  in Example~\ref{eg:W:values4d}. Observe that, since the first column of $A (
  \nu)$ is composed of $1$'s, the sequence of moments $W_n ( 1 ; 2 k)$ can
  also be directly read off from the first column of $A ( \nu)^n$. The matrix
  $A (1)$ is known as the {\emph{Narayana triangle}} or the {\emph{Catalan
  triangle}} \cite[\texttt{A001263}]{sloane-oeis}.
\end{example}

\begin{remark}
  Let us note another point of view on the appearance of the Narayana triangle
  in the context of random walks. Let $\boldsymbol{X}$ be a random vector, which
  is uniformly distributed on the unit sphere in $\mathbb{R}^d$. If $\theta$
  is the angle between $\boldsymbol{X}$ and a fixed axis, then $\Lambda = \cos
  \theta$ has the probability density \cite[(2)]{kingman-rw}
  \begin{equation*}
    \frac{\nu !}{\sqrt{\pi}  (\nu - 1 / 2) !}  (1 - \lambda^2)^{\nu - 1 / 2},
     \hspace{1em} \lambda \in [- 1, 1] .
  \end{equation*}
  Denote with $R_n$ the random variable describing the distance to the origin
  after $n$ unit steps. Then $R_{n + 1}$ is related to $R_n$ via
  \cite[(9)]{kingman-rw}
  \begin{equation*}
    R_{n + 1} = \sqrt{1 + 2 \Lambda R_n + R_n^2} .
  \end{equation*}
  Writing $\mathbb{E} [X]$ for the expected value of a random variable $X$,
  we therefore have
  \begin{equation*}
    W_{n + 1} (\nu ; 2 k) =\mathbb{E} [R_{n + 1}^{2 k}] =\mathbb{E} [(1 + 2
     \Lambda R_n + R_n^2)^k] .
  \end{equation*}
  In terms of the generalized {\emph{Narayana polynomials}}
  \begin{equation}
    \mathcal{N}_k^{(\nu)} (z) =\mathbb{E} [(1 + 2 \lambda \sqrt{z} + z)^{k -
    1}], \label{eq:narayana:poly}
  \end{equation}
  which were introduced in \cite[(6.2)]{amv-narayana} for $k \geq 1$,
  we obtain
  \begin{equation*}
    W_{n + 1} (\nu ; 2 k) =\mathbb{E} [\mathcal{N}_{k + 1}^{(\nu)}
     (R_n^2)_{}] .
  \end{equation*}
  This recurrence identity on the moments can then be expressed in matrix form
  by defining a matrix $A (\nu)$ as in Example~\ref{eg:W:matpow}.

  We note, moreover, that expressing the Narayana polynomials
  \eqref{eq:narayana:poly} in terms of the Gegenbauer polynomials $C_k^{\mu}
  (z)$, as demonstrated in \cite[Theorem~6.3]{amv-narayana}, we deduce the
  expression of the even moments as
  \begin{equation*}
    W_{n + 1} (\nu ; 2 k) = \frac{k!}{(2 \nu + 1)_k} \mathbb{E} \left[ (1 -
     R_n^2)^k C_k^{\nu + 1 / 2} \left( \frac{1 + R_n^2}{1 - R_n^2} \right)
     \right],
  \end{equation*}
  which is a variation of \eqref{eq:W:crec1}.
\end{remark}

\begin{example}
  {\dueto{Six dimensional even moments}}To contrast with the integral even
  moments in 2 and 4 dimensions in Example~\ref{eg:W:values4d}, we record a
  few initial even moments in 6 dimensions.
  \begin{eqnarray*}
    W_2 ( 2 ; 2 k) & : & 1, 2, 14 / 3, 12, 33, 286 / 3, 286, 884, 8398 / 3,
    \ldots\\
    W_3 ( 2 ; 2 k) & : & 1, 3, 11, 139 / 3, 216, 1088, 5825, 32763, 191935,
    \ldots\\
    W_4 ( 2 ; 2 k) & : & 1, 4, 20, 352 / 3, 2330 / 3, 16952 / 3, 133084 / 3,
    370752, 3265208, \ldots
  \end{eqnarray*}
  It may be concluded from \eqref{eq:A:kj}, with $\nu = 2$, that the entries
  of the matrix $A (2)$ satisfy $A_{k, j} (2) \in \frac{1}{3}
  \mathbb{Z}$. This implies that the even moments in dimension 6 are rational
  numbers whose denominators are powers of $3$. Similar observations apply in
  all dimensions but we do not pursue this theme further here.
\end{example}

\section{Moments of short walks}\label{sec:mom}

\subsection{Moments of $2$-step walks}\label{ssec:mom2}

It follows from Theorem~\ref{thm:W:even} that the general expression of the
even moments for a $2$-step walk in $d$ dimensions is given by
\begin{equation}
  W_2 ( \nu ; 2 k) = \frac{\binom{2 k + 2 \nu}{k}}{\binom{k + \nu}{k}} =
  \frac{\nu ! ( 2 k + 2 \nu) !}{( k + \nu) ! ( k + 2 \nu) !} . \label{eq:W2}
\end{equation}

\begin{example}
  Note that in the special case of $\nu = 0$, that is, dimension $2$, this
  clearly reduces to the central binomial coefficient. In dimension $4$, as
  noted in Example~\ref{eg:W:values4d}, the even moments are the Catalan
  numbers. We note that the generating function for the two-step even moments
  is
  \begin{equation}
    \sum_{k = 0}^{\infty} W_2 ( \nu ; 2 k) x^k =\pFq21{1, \nu + \tfrac{1}{2}}{2 \nu + 1}{4 x}, \label{eq:W2:gf}
  \end{equation}
  which reduces to the known generating functions for $\nu = 0, 1$. Equation
  \eqref{eq:W2:gf} is an immediate consequence of rewriting \eqref{eq:W2} as
  $W_2 ( \nu ; 2 k) = 2^{2 k} \frac{( \nu + 1 / 2)_k}{( 2 \nu + 1)_k}$.
\end{example}

In fact, \eqref{eq:W2} also holds true when $k$ takes complex values. This was
proved in \cite[Theorem~4.3]{wan-phd} using a multi-dimensional integral
representation and hyper-spherical coordinates. We offer an alternative proof
based on Theorem~\ref{thm:W:bessel}.

\begin{theorem}
  \label{thm:W2}For all complex $s$ and half-integer $\nu \ge0$,
  \begin{equation*}
    W_2 ( \nu ; s) = \frac{\nu ! \Gamma ( s + 2 \nu + 1)}{\Gamma \left(
     \frac{s}{2} + \nu + 1 \right) \Gamma \left( \frac{s}{2} + 2 \nu + 1
     \right)}.
  \end{equation*}
\end{theorem}

\begin{proof}
  The case $k = 0$ of \eqref{eq:W:bessel} in Theorem~\ref{thm:W:bessel} gives
  \begin{equation*}
    W_n ( \nu ; s) = \frac{2^{s + 1} \Gamma \left( \frac{s}{2} + \nu + 1
     \right)}{\Gamma (\nu + 1) \Gamma \left( - \frac{s}{2} \right)}
     \int_0^{\infty} x^{- s - 1} j_{\nu}^n ( x) \mathd x,
  \end{equation*}
  provided that $- n (\nu + 1 / 2) - 1 < s < 0$. Using that, for $0 <
  \Re s < 2 \nu + 1$,
  \begin{equation}
    \int_0^{\infty} x^{s - 1} j_{\nu}^2 ( x) \mathd x = 2^{2 \nu - 1}
    \frac{\Gamma \left( \frac{s}{2} \right) \Gamma \left( \frac{1 - s}{2} +
    \nu \right) \Gamma ( 1 + \nu)^2}{\Gamma \left( \frac{1}{2} \right) \Gamma
    \left( 1 - \frac{s}{2} + \nu \right) \Gamma \left( 1 - \frac{s}{2} + 2 \nu
    \right)}, \label{eq:j2:mellin}
  \end{equation}
  the claimed formula then follows from the duplication formula for the gamma
  function and analytic continuation.
\end{proof}

\subsection{Moments of $3$-step walks}\label{ssec:mom3}

\begin{lemma}
  The nonnegative even moments for a $3$-step walk in $d$ dimensions are
  \begin{equation}
    W_3 ( \nu ; 2 k) = \sum_{j = 0}^k \binom{k}{j} \binom{k + \nu}{j} \binom{2
    j + 2 \nu}{j} \binom{j + \nu}{j}^{- 2} . \label{eq:W3}
  \end{equation}
\end{lemma}

\begin{proof}
  We apply Corollary~\ref{cor:W:iter} with $n_1 = 2$ and $n_2 = 1$. Using that
  $W_1 ( \nu ; 2 k) = 1$ and that an evaluation of $W_2 ( \nu ; 2 k)$ is given
  by \eqref{eq:W2}, we obtain
  \begin{eqnarray*}
    W_3 ( \nu ; 2 k) & = & \sum_{j = 0}^k \binom{k}{j} \frac{( k + \nu) ! \nu
    !}{( k - j + \nu) ! ( j + \nu) !} W_2 ( \nu ; 2 j)\\
    & = & \sum_{j = 0}^k \binom{k}{j} \frac{( k + \nu) ! \nu !}{( k - j +
    \nu) ! ( j + \nu) !} \frac{\nu ! ( 2 j + 2 \nu) !}{( j + \nu) ! ( j + 2
    \nu) !} .
  \end{eqnarray*}
  Expressing the factorials as binomial coefficients yields \eqref{eq:W3}.
\end{proof}

\begin{example} \dueto{Generating function for 3 steps in 2 dimensions}
  In the case $d = 2$, or $\nu = 0$, the moments of a $3$-step walk reduce to
  the Ap\'ery-like numbers
  \begin{equation}
    W_3 ( 0 ; 2 k) = \sum_{j = 0}^k \binom{k}{j}^2 \binom{2 j}{j} .
    \label{eq:W3:d2}
  \end{equation}
  \cite[(3.2) \& (3.4)]{bswz-densities} show that the generating function
  for this sequence is
  \begin{equation*}
    \sum_{k = 0}^{\infty} W_3 ( 0 ; 2 k) x^k = \frac{1}{1 + 3 x} \pFq21{\tfrac{1}{3}, \tfrac{2}{3}}{1}{\frac{27 x ( 1 - x)^2}{( 1 + 3 x)^3}} .
  \end{equation*}
\end{example}

\begin{example}\dueto{Generating function for 3 steps in 4 dimensions}
  \label{eg:W3:4d}In the case $d = 4$, or $\nu = 1$, the moments, whose
  initial values are recorded in Example~\ref{eg:W:values4d}, are sequence
  \cite[\texttt{A103370}]{sloane-oeis}. The OEIS also records a hypergeometric form of the generating
  function (as the linear combination of a hypergeometric function and its
  derivative), added by Mark van Hoeij. On using linear transformations of hypergeometric functions, we
  have more simply that
  \begin{equation*}
    \sum_{k = 0}^{\infty} W_3 ( 1 ; 2 k) x^k = \frac{1}{2 x^2} - \frac{1}{x}
     - \frac{( 1 - x)^2}{2 x^2 ( 1 + 3 x)} \pFq21{\tfrac{1}{3}, \tfrac{2}{3}}{2}{\frac{27 x ( 1 - x)^2}{( 1 + 3 x)^3}} ,
  \end{equation*}
  which we are able to generalize.
\end{example}

Example~\ref{eg:W3:4d} suggests that a nice formula for the generating function for
the moments $W_3 ( \nu ; 2 k)$ exists for all even dimensions. Indeed, we have the
following result.

\begin{theorem}\dueto{Ordinary generating function for even moments with three steps}
  \label{thm:W3:gf}For integers $\nu \geq 0$ and $|x| < 1/9$, we have
  \begin{equation}
    \sum_{k = 0}^{\infty} W_3 ( \nu ; 2 k) x^k = \frac{( - 1)^{\nu}}{\binom{2
    \nu}{\nu}}  \frac{( 1 - 1 / x)^{2 \nu}}{1 + 3 x} \pFq21{\tfrac{1}{3}, \tfrac{2}{3}}{1 + \nu}{\frac{27 x ( 1 - x)^2}{( 1 + 3 x)^3}} - q_\nu \left( \frac1 x\right),
    \label{eq:W3:gf}
  \end{equation}
  where $q_\nu( x)$ is a polynomial (that is, $q_\nu ( 1 / x)$ is the principal part
  of the hypergeometric term on the right-hand side).
\end{theorem}

\begin{proof}
  For integers $\nu \geq 0$, define
  the rational numbers $H (\nu ; k)$ by
  \begin{equation}
    \frac{( - 1)^{\nu}}{\binom{2 \nu}{\nu}}  \frac{( 1 - 1 / x)^{2 \nu}}{1 + 3
    x} \pFq21{\tfrac{1}{3}, \tfrac{2}{3}}{1 + \nu}{\frac{27 x ( 1 - x)^2}{( 1 + 3 x)^3}} = \sum_{k = - 2
    \nu}^{\infty} H (\nu ; k) x^k \label{eq:W3:gfx}
  \end{equation}
  for $k \geq - 2 \nu$, and $H (\nu ; k) = 0$ for $k < - 2 \nu$. Writing
  the sum \eqref{eq:W3} for $W_3 ( \nu ; 2 k)$ in hypergeometric form, we
  obtain, for integers $k, \nu \geq 0$, the representation
  \begin{equation}
    W_3 ( \nu ; 2 k) =\pFq32{- k, - k - \nu, \nu + 1 / 2}{\nu + 1, 2 \nu + 1}{4} . \label{eq:W3:3F2:even}
  \end{equation}
  In order to prove the claimed generating function \eqref{eq:W3:gf} it
  therefore suffices to show the (more precise) claim
  \begin{equation}
    H (\nu ; k) = \Re \pFq32{- k, - k - \nu, \nu + 1 / 2}{\nu + 1, 2 \nu + 1}{4} . \label{eq:W3:H}
  \end{equation}
  For instance, this predicts that $q_2(x)= 1/6-5/6\,{x}+{x}^{2}+1/3\,{x}^{3}+{x}^{4}$.

  Taking the real part is only necessary for $k < - \nu$ while, for $k
  \geq - \nu$, the ${}_3 F_2$ is terminating. We note, as will be
  demonstrated later in the proof, that the right-hand side of \eqref{eq:W3:H}
  vanishes for $k < - 2 \nu$, that is, the ${}_3 F_2$ takes purely imaginary
  values then.

  The holonomic systems approach \cite{zeilberger90}, implemented in the \emph{Mathematica} package
  \texttt{Holo\-nomic\-Functions}, which accompanies Koutschan's thesis
  \cite{koutschan-phd}, shows that the coefficients $H (\nu ; k)$, defined
  by \eqref{eq:W3:gfx}, satisfy the recursive relation
  \begin{align}
    9 ( k + 1) ( k + \nu + 1) H (\nu ; k) & = \tfrac{1}{2} \left( 20 \left(
    k + \tfrac{3}{2} \right)^2 + 60 \left( k + \tfrac{3}{2} \right) \nu + 36
    \nu^2 + 1 \right) H (\nu ; k + 1) \nonumber\\
    &\quad - ( k + 2 \nu + 2) ( k + 3 \nu + 2) H (\nu ; k + 2)
    \label{eq:W3:H:rec:k}
  \end{align}
  for all integers $k$ and all integers $\nu \geq 0$. We already know
  that \eqref{eq:W3:H} holds for $\nu = 0$ and $k \geq 0$. Verifying, by
  using
  \begin{equation*}
    \pFq32{1, 1, 1 / 2}{1, 1}{4 x} = \frac{1}{\sqrt{1 - 4 x}}
  \end{equation*}
  and letting $x \rightarrow 1$ to see that the real part vanishes (for any
  choice of analytic continuation to $x = 1$), that \eqref{eq:W3:H} holds for
  $\nu = 0$ and $k = - 1$, we conclude from \eqref{eq:W3:H:rec:k} that
  \eqref{eq:W3:H} is indeed true for $\nu = 0$ and all integers $k$.

  As in the case of \eqref{eq:W3:H:rec:k}, we find that the coefficients $H
  (\nu ; k)$ further satisfy the dimensional relations
  \begin{eqnarray}
    &  & 2 (k + 2 \nu) (k + 3 \nu - 1) (k + 3 \nu) H (\nu ; k) \nonumber\\
    & = & \nu^2 (7 k + 15 \nu - 4) H (\nu - 1 ; k + 1) - 9 \nu^2 (k + \nu) H
    (\nu - 1 ; k)  \label{eq:W3:H:rec1}
  \end{eqnarray}
  and
  \begin{equation}
    2 (k + 1) (k + 3 \nu) H (\nu ; k) = \nu^2 H (\nu - 1 ; k + 2) - 3 \nu^2 H
    (\nu - 1 ; k + 1) \label{eq:W3:H:rec2}
  \end{equation}
  for all integers $k$ and all integers $\nu \geq 1$. The three
  relations \eqref{eq:W3:H:rec:k}, \eqref{eq:W3:H:rec1}, \eqref{eq:W3:H:rec2},
  together with the case $\nu = 0$ as boundary values, completely determine
  the coefficients $H (\nu ; k)$ for all integers $k$ and $\nu \geq
  1$.

  Since we already verified the case $\nu = 0$, it only remains to demonstrate
  that the right-hand side of \eqref{eq:W3:H} satisfies the same recursive
  relations. Another application of \texttt{Holo\-nomic\-Functions} finds that
  the ${}_3 F_2$ on right-hand side of \eqref{eq:W3:H}, and hence its real part,
  indeed satisfy \eqref{eq:W3:H:rec:k}, \eqref{eq:W3:H:rec1},
  \eqref{eq:W3:H:rec2} for the required (real) values of $\nu$ and $k$.
\end{proof}

For the convenience of the reader, and because we will frequently use it in
the following, we state Carlson's Theorem next
\cite[5.81]{titchmarsh-theoryoffunctions}. Recall that a function $f ( z)$
is of {\emph{exponential type}} in a region if $|f (z) | \le Me^{c |z|}$ for
some constants $M$ and $c$.

\begin{theorem}
  {\dueto{Carlson's Theorem}}\label{thm:carlson}Let $f$ be analytic in the
  right half-plane $\Re z \geq 0$ and of exponential type with the
  additional requirement that
  \begin{equation*}
    |f (z) | \leq Me^{d |z|}
  \end{equation*}
  for some $d < \pi$ on the imaginary axis $\Re z = 0$. If $f (k) = 0$
  for $k = 0, 1, 2, \ldots$, then $f (z) = 0$ identically.
\end{theorem}

\begin{example}
  \label{eg:rec3}Applying creative telescoping to the binomial sum
  \eqref{eq:W3}, we derive that the moments $W_3 ( \nu ; 2 k)$ satisfy the
  recursion
  \begin{eqnarray}
    &  & ( k + 2 \nu + 1) ( k + 3 \nu + 1) W_3 ( \nu ; 2 k + 2) \label{eq:rec3}\\
    & = & \tfrac{1}{2} \left( 20 \left( k + \tfrac{1}{2} \right)^2 + 60
    \left( k + \tfrac{1}{2} \right) \nu + 36 \nu^2 + 1 \right) W_3 ( \nu ; 2
    k) - 9 k ( k + \nu) W_3 ( \nu ; 2 k - 2) .  \nonumber
  \end{eqnarray}
  Observe that $W_n ( \nu ; s)$ is analytic for $\Re s \geq 0$ and
  is bounded in that half-plane by $| W_n ( \nu ; s) | \leq n^{\Re s}$ (because, in any dimension, the distance after $n$ random steps is
  bounded by $n$). It therefore follows from Carlson's Theorem, as detailed in
  \cite[Theorem~4]{bnsw-rw}, that the recursion \eqref{eq:rec3} remains
  valid for complex values of $k$.
\end{example}

The next result expresses the complex moments $W_3 (\nu ; s)$ in terms of a
Meijer $G$-function and extends \cite[Theorem 2.7]{bsw-rw2}.

\begin{theorem}\dueto{Meijer $G$ form of $W_3$}
  \label{thm:W3:meijer}For all complex $s$ and dimensions $d \geq 2$,
  \begin{equation*}
    W_3 ( \nu ; s) = 2^{2 \nu} \nu !^2 \frac{\Gamma \left( \frac{s}{2} + \nu
     + 1 \right)}{\Gamma \left( \frac{1}{2} \right) \Gamma \left( -
     \frac{s}{2} \right)} \MeijerG{2, 1}{3, 3}{1, 1 + \nu, 1 + 2 \nu}{\tfrac{1}{2} + \nu, - \tfrac{s}{2}, - \tfrac{s}{2} - \nu}{\frac{1}{4}} .
  \end{equation*}
\end{theorem}

\begin{proof}
  If $0 < \Re ( s) < \nu + \frac{3}{2}$, then
  \begin{equation*}
    \int_0^{\infty} x^{s - 1} j_{\nu} ( x) \mathd x = 2^{s - 1} \frac{\Gamma
     \left( \frac{s}{2} \right) \Gamma ( 1 + \nu)}{\Gamma \left( 1 -
     \frac{s}{2} + \nu \right)} .
  \end{equation*}
  This simple integral is a consequence of the fact that $W_1 (\nu ; s) = 1$
  combined with Theorem~\ref{thm:W:bessel} with $n = 1$ and $k = 0$.
  Similarly, for $0 < \Re ( s) < 2 \nu + 1$, the Mellin transform of
  $j_{\nu}^2 (x)$ is given by \eqref{eq:j2:mellin}. Applying Parseval's
  formula to these two Mellin transforms, we obtain, for $0 < \delta < 1$,
  \begin{eqnarray*}
    &  & \int_0^{\infty} x^{s - 1} j_{\nu}^3 ( x) \mathd x\\
    & = & \frac{1}{2 \pi i} \int_{\delta - i \infty}^{\delta + i \infty} 2^{2
    \nu - 1} \frac{\Gamma \left( \frac{z}{2} \right) \Gamma \left( \frac{1 -
    z}{2} + \nu \right) \Gamma ( 1 + \nu)^2}{\Gamma \left( \frac{1}{2} \right)
    \Gamma \left( 1 - \frac{z}{2} + \nu \right) \Gamma \left( 1 - \frac{z}{2}
    + 2 \nu \right)} 2^{s - z - 1} \frac{\Gamma \left( \frac{s - z}{2} \right)
    \Gamma ( 1 + \nu)}{\Gamma \left( 1 - \frac{s - z}{2} + \nu \right)} \mathd
    z\\
    & = & \frac{2^{2 \nu + s - 1} \Gamma ( 1 + \nu)^3}{\Gamma \left(
    \frac{1}{2} \right)}  \frac{1}{2 \pi i} \int_{\delta / 2 - i
    \infty}^{\delta / 2 + i \infty} \frac{2^{- 2 t} \Gamma ( t) \Gamma \left(
    \frac{1}{2} - t + \nu \right)}{\Gamma ( 1 - t + \nu) \Gamma ( 1 - t + 2
    \nu)} \frac{\Gamma \left( \frac{s}{2} - t \right)}{\Gamma \left( 1 -
    \frac{s}{2} + t + \nu \right)} \mathd t\\
    & = & \frac{2^{2 \nu + s - 1} \nu !^3}{\Gamma \left( \frac{1}{2} \right)}
    \MeijerG{2, 1}{3, 3}{1, 1 + \nu, 1 + 2 \nu}{\tfrac{1}{2} + \nu, \tfrac{s}{2}, \tfrac{s}{2} - \nu}{\frac{1}{4}} .
  \end{eqnarray*}
  The claim then follows, by analytic continuation, from
  Theorem~\ref{thm:W:bessel} with $n = 3$ and $k = 0$.
\end{proof}

We note that, as in \cite{bsw-rw2}, this Meijer $G$-function expression can
be expressed as a sum of hypergeometric functions by Slater's Theorem
\cite[p.~57]{mar}. This is made explicit in \eqref{eq:W3:meijer:hyp}.

Equation~\eqref{eq:W3:3F2:even} gives a hypergeometric expression for the even
moments of a $3$-step random walk. It was noticed in \cite{bnsw-rw} that, in
the case of planar walks, the {\emph{real part}} of this hypergeometric
expression provides an evaluation of the odd moments. These odd moments are
much harder to obtain, and it was first proved in \cite{bnsw-rw}, based on
this observation, that the average distance of a planar $3$-step random walk
is
\begin{equation}
  W_3 (0 ; 1) = A + \frac{6}{\pi^2}  \frac{1}{A} \approx 1.5746,
  \label{eq:W3:0:1}
\end{equation}
where
\begin{equation}
  W_3 (0 ; - 1) = \frac{3}{16}  \frac{2^{1 / 3}}{\pi^4} \Gamma^6 \left(
  \tfrac{1}{3} \right) = : A. \label{eq:W3:0:-1}
\end{equation}
In the sequel, we generalize these results from two to arbitrary even
dimensions. In particular, as explained in Example~\ref{eg:W3:odd}, we
establish the transcendental nature of the odd moments of $3$-step walks in
all even dimensions by showing that they are all rational linear combinations
of $A$ and $1 / (\pi^2 A)$.

\begin{theorem}\dueto{Hypergeometric form of $W_3$ at odd integers}
  \label{thm:W3:Re}Suppose that $d$ is even, that is, $\nu$ is an integer. For
  all odd integers $s \geq - 2 \nu - 1$,
  \begin{equation*}
    W_3 ( \nu ; s) = \Re \pFq32{- s / 2, - s / 2 - \nu, \nu + 1 / 2}{\nu + 1, 2 \nu + 1}{4} .
  \end{equation*}
\end{theorem}

\begin{proof}
  The case $\nu = 0$ is proved in \cite[Theorem~6]{bnsw-rw}. We will prove
  the general case by induction on $\nu$.

  It is routine to verify that the hypergeometric function
  \begin{equation*}
    F ( \nu ; s) =\pFq32{- s / 2, - s / 2 - \nu, \nu + 1 / 2}{\nu + 1, 2 \nu + 1}{4},
  \end{equation*}
  which, for even $s$, agrees with $W_3 ( \nu ; s)$ by
  \eqref{eq:W3:3F2:even}, satisfies the contiguity relation
  \begin{equation*}
    ( s + 1) ( s + 6 \nu - 1) F ( \nu ; s - 1) + 6 \nu^2 F ( \nu - 1 ; s + 1)
     - 2 \nu^2 F ( \nu - 1 ; s + 3) = 0.
  \end{equation*}
  On the other hand, it follows from Theorem~\ref{thm:W3:meijer} and Slater's
  Theorem \cite[p.~57]{mar} that
  \begin{eqnarray}
    W_3 ( \nu ; s) & = & \frac{\Gamma \left( - \nu - \tfrac{s + 1}{2} \right)
    \Gamma ( \nu + 1)^2 \Gamma \left( \nu + \tfrac{s}{2} + 1 \right)}{2 \pi
    \Gamma \left( - \tfrac{s}{2} \right) \Gamma \left( 2 \nu + \tfrac{s +
    3}{2} \right)} \pFq32{\tfrac{1}{2}, \tfrac{1}{2} - \nu, \tfrac{1}{2} + \nu}{\nu + \tfrac{s + 3}{2}, 2 \nu + \tfrac{s + 3}{2}}{\frac{1}{4}} \nonumber\\
    &  & + \frac{2^{s + 2 \nu} \Gamma ( \nu + 1) \Gamma \left( \nu + \tfrac{s
    + 1}{2} \right)}{\sqrt{\pi} \Gamma \left( 2 \nu + \tfrac{s}{2} + 1
    \right)} \pFq32{- \tfrac{s}{2}, - \tfrac{s}{2} - \nu, - \tfrac{s}{2} - 2 \nu}{\nu + 1, - \nu - \tfrac{s - 1}{2}}{\frac{1}{4}} .  \label{eq:W3:meijer:hyp}
  \end{eqnarray}
  In that form, it is again a routine application of the holonomic systems
  approach \cite{zeilberger90} to derive that
  \begin{eqnarray}
    0 & = & ( s + 1) ( s + 6 \nu - 1) W_3 ( \nu ; s - 1) \nonumber\\
    &  & + 6 \nu^2 W_3 ( \nu - 1 ; s + 1) - 2 \nu^2 W_3 ( \nu - 1 ; s + 3) .
    \label{eq:W3:rec:nu:s}
  \end{eqnarray}
  Since this relation matches the relation satisfied by $F (\nu ; s)$, and
  hence $\Re F (\nu ; s)$ when $\nu$ and $s$ are real, the general case
  follows inductively from the base case $\nu = 0$.
\end{proof}

\begin{remark}
  \label{rk:W3:meijer:hyp}The coefficients of the hypergeometric functions in
  \eqref{eq:W3:meijer:hyp} can be expressed as
  \begin{equation*}
    c_1 \assign \frac{\tau (\nu ; s)}{\pi}  \frac{2^{2 \nu - 1}}{( 2 \nu + 1)
     \binom{2 \nu}{\nu}}  \frac{\binom{2 \nu + s + 1}{\nu + (s + 1) /
     2}}{\binom{s + 2 \nu}{s / 2} \hspace{0.25em} \binom{4 \nu + s + 1}{2 \nu
     + 1}}
  \end{equation*}
  and $c_2 \assign \binom{2 \nu + s}{\nu + s / 2} / \binom{2 \nu + s / 2}{\nu}$,
  respectively. Here, the factor
  \begin{equation*}
    \tau (\nu ; s) = \frac{1}{\cos ( \pi \nu) \cot ( \pi s / 2) - \sin (\pi
     \nu)}
  \end{equation*}
  is $\pm 1$ for half-integers $\nu$, and $\pm \tan ( \frac{\pi s}{2})$ for
  integers $\nu$.
\end{remark}

\begin{example}\dueto{Odd moments $W_3(\nu;\cdot)$ in even dimensions}
  \label{eg:W3:odd}The planar case of Theorem~\ref{thm:W3:Re} was used in
  \cite{bnsw-rw} to prove that the average distance to the origin after three random
  steps in the plane is given by \eqref{eq:W3:0:1}. It is a consequence of
  \eqref{eq:rec3}, extended to complex $k$, that all planar odd moments are
  $\mathbb{Q}$-linear combinations of $A = \frac{3}{16}  \frac{2^{1 /
  3}}{\pi^4} \Gamma^6 ( 1 / 3)$, defined in \eqref{eq:W3:0:-1}, and $1 / (
  \pi^2 A)$.

  The dimensional recursion \eqref{eq:W3:rec:nu:s} used in the
  proof of Theorem~\ref{thm:W3:Re} shows that this observation extends to all
  even dimensions. For instance,
  \begin{equation*}
    W_3 (1 ; - 3) = \frac{4}{3} A - \frac{4}{\pi^2}  \frac{1}{A},
     \hspace{1em} W_3 (1 ; - 1) = \frac{4}{15} A + \frac{4}{\pi^2}
     \frac{1}{A} .
  \end{equation*}
  Moreover, the average distance to the origin after three random steps in four
  dimensions is
  \begin{equation*}
    W_3 (1 ; 1) = \frac{476}{525} A + \frac{52}{7 \pi^2}  \frac{1}{A} \approx
     1.6524,
  \end{equation*}
  with similar evaluations in six or higher even dimensions.
\end{example}

\begin{example}
  Theorem~\ref{thm:W3:Re} does not hold in odd dimensions, in which the
  involved quantities can be evaluated in elementary terms. For instance, in
  the case of dimension $3$,
  \begin{equation*}
    W_3 ( \tfrac{1}{2} ; s) = \frac{1}{4}  \frac{3^{s + 3} - 3}{(s + 2) (s +
     3)}
  \end{equation*}
  while for integer $s>0$
  \begin{equation*}
    \pFq32{- s / 2, - s / 2 - 1 / 2, 1}{3 / 2, 2}{4} = \frac{1}{4}  \frac{3^{s + 3} - 2 - (- 1)^s}{(s
     + 2) (s + 3)}
  \end{equation*}
  only agrees for even $s$.
\end{example}

\begin{example}\dueto{First derivative of $W_3(\nu;\cdot)$ in even dimensions}
  \label{eg:W3:der}By differentiating the hypergeometric representation of
  $W_3 (\nu ; s)$ in \eqref{eq:W3:meijer:hyp}, it was shown in \cite[Examples
  6.2 and 6.6]{bswz-densities} that
  \begin{equation*}
    W_3' (0 ; 0) = \frac{1}{\pi} \operatorname{Cl} \left( \frac{\pi}{3} \right),
     \hspace{1em} W_3' (0 ; 2) = 2 + \frac{3}{\pi} \operatorname{Cl} \left(
     \frac{\pi}{3} \right) - \frac{3 \sqrt{3}}{2 \pi},
  \end{equation*}
  where the derivatives are with respect to $s$. It follows from
  differentiating \eqref{eq:rec3}, extended to complex $k$, that all
  derivatives $W_3' (0 ; 2 k)$ lie in the $\mathbb{Q}$-linear span of $1$,
  $\frac{1}{\pi} \operatorname{Cl} ( \frac{\pi}{3})$ and $\frac{\sqrt{3}}{\pi}$. Then,
  differentiating \eqref{eq:W3:rec:nu:s}, we find that, indeed, for all
  integers $\nu \geq 0$, the derivatives $W_3' (\nu ; 2 k)$ can likewise
  be expressed as
  \begin{equation*}
    W_3' (\nu ; 2 k) = r_1 + r_2 \frac{\sqrt{3}}{\pi} + r_3 \frac{1}{\pi}
     \operatorname{Cl} ( \frac{\pi}{3}),
  \end{equation*}
  with rational numbers $r_1, r_2, r_3$.

  Moreover, in the case of $W_3' (\nu ; 0)$, a slightly more careful analysis
  reveals that $r_3 = 1$. While we omit the details, we note that this can be
  seen, for instance, by evaluating $W_3' (1 ; 0)$ and then deriving, in
  analogy with \eqref{eq:W3:rec:nu:s}, a functional equation relating $W_3
  (\nu ; s)$, $W_3 (\nu + 1 ; s)$ and $W_3 (\nu + 2 ; s)$. In four and six
  dimensions, we obtain, for example,
  \begin{equation*}
    W_3' (1 ; 0) = \frac{1}{2} - \frac{11 \sqrt{3}}{16 \pi} + \frac{1}{\pi}
     \operatorname{Cl} \left( \frac{\pi}{3} \right), \hspace{1em} W_3' (2 ; 0) =
     \frac{17}{36} - \frac{181 \sqrt{3}}{320 \pi} + \frac{1}{\pi} \operatorname{Cl}
     \left( \frac{\pi}{3} \right) .
  \end{equation*}
  A special motivation for considering these derivative values is that, in the
  case of two dimensions, $W_3' (0 ; 0)$ is the (logarithmic) Mahler measure
  of the multivariate polynomial $1 + x_1 + x_2$; see \cite[Example
  6.6]{bswz-densities} or \cite[Section~4]{logsin1}.
\end{example}

\begin{example}\dueto{Second derivative of $W_3(\nu;\cdot)$ in even dimensions}
  The second derivative $W_3'' ( 0 ; 0)$, interpreted there as a higher Mahler
  measure, is evaluated in \cite[Theorem~4.4]{logsin1} in the form
    \begin{equation}
    W_3'' ( 0 ; 0) = \frac{\pi^2}{4} + \frac{3}{\pi} \operatorname{Ls}_3 \left(
    \frac{2 \pi}{3} \right), \label{eq:W3dd00}
  \end{equation}
  where $\operatorname{Ls}_n$ denotes the $n$th {\emph{log-sine integral}}
  \begin{equation*}
    \operatorname{Ls}_n ( \sigma) \assign - \int_0^{\sigma} \log^{n - 1} \left| 2
     \sin \frac{\theta}{2} \right| \mathd \theta .
  \end{equation*}
  For alternative expressions of the log-sine integral in \eqref{eq:W3dd00} in
  terms of other polylogarithmic constants, we refer to \cite{logsin1}. 
  We have not been able to obtain an equally natural log-sine evaluation of
\begin{equation}
  W_3'' (0 ;2k)=\frac{1}{16 \pi^2}\int_0^{2\pi}\int_0^{2\pi}|1+e^{i\omega}+e^{i\theta}|^{2k}\log^2 |1+e^{i\omega}+e^{i\theta}| \,\mathd \theta \mathd\omega,
\end{equation}
for integers $k>0$. We may, however, derive
\begin{equation}
  W_3'' (0 ;2)=3\,W_3'' (0 ;0)+ \frac{3}{ \pi^2}\int_0^\pi \cos \omega \,\Re\int_0^{2\pi}\log^2 \left(1-2\sin \left(\frac \omega2\right) e^{i\theta}\right) \,\mathd \theta\mathd\omega.
\end{equation}
By the methods of \cite[\S4]{logsin1}, we arrive at
\begin{equation}
  W_3'' (0 ;2)=3\,W_3'' (0 ;0)+\frac{\sqrt {3}}2\,\pi-\frac32-\frac{6}\pi\,\int_{0}^{\pi /3}\!{\rm Li
}_2 \left( 4\,  \sin^2 \left( \frac \omega 2 \right)   
 \right)  \cos \omega\,{\mathd}\omega.
\end{equation}
We may now integrate by parts and obtain
\begin{equation}
  W_3'' (0 ;2)=3\,W_3'' (0 ;0)-1-{\frac {2\sqrt {3}}{\pi }\int_{0}^{1}\!\sqrt {{\frac {1+\frac s3}{1-s}}}
\log   s \,\mathd s}.
\end{equation}
Moreover,
\begin{equation}
\int_{0}^{1}\!\sqrt {\frac {1+\frac s3}{1-s}}
\log   s \,\mathd s = \sum_{n=0}^\infty \frac {a_n}{3^n\,n^2}
\end{equation}
where $a_n$ is given by \cite[\texttt{A025565}]{sloane-oeis} of the OEIS and counts ``the number of number of UDU-free paths of $n-1$ upsteps (U) and $n-1$ downsteps (D)" with recursion
\[(n-1)a_n-2(n-1)a_{n-1}-3(n-3)a_{n-2}=0.\]
(Alternatively, $a_n = M_{n-1} + \sum_{k=1}^{n-1} M_{k-1}a_{n-k}$ with $M_k$ the \emph{Motzkin numbers} given in \texttt{A001006}.) Solving for the generating function $a_3(x)$ of $a_n$ and considering $\int_0^{1/2}\int_0^t a_3(x) \mathd x\frac{\mathd t}t$ we finally arrive at
\begin{equation}W_3'' (0 ;2)=3\,W_3'' (0 ;0)-3\,\frac {\sqrt {3}}{\pi } (\log 3-1)-\frac12+\frac 4 {\pi }{ \operatorname{Cl}
     \left( \frac{\pi}{3} \right)}
.\end{equation}
  As in Example~\ref{eg:W3:der}, all second derivatives $W_3'' ( \nu
  ; 2 k)$ in even dimensions may then be expressed in terms of $W_3'' ( 0 ;
  0)$ and $W_3'' ( 0 ; 2)$ as well as the constants in
  Example~\ref{eg:W3:der}. For instance,
  \begin{equation*}
    W_3'' (1 ; 0) = - \frac{3}{8} W_3'' (0 ; 0) + \frac{11}{24} W_3'' (0 ; 2)
     - \frac{3}{8} + \frac{23 \sqrt{3}}{48 \pi} - \frac{5}{6 \pi} \operatorname{Cl}
     \left( \frac{\pi}{3} \right) ,
  \end{equation*}
  and so on.
\end{example}

\subsection{Moments of $4$-step walks}\label{ssec:mom4}

\begin{lemma}
  The nonnegative even moments for a $4$-step walk in $d$ dimensions are
  \begin{equation}
    W_4 ( \nu ; 2 k) = \sum_{j = 0}^k \frac{\binom{k}{j} \binom{k +
    \nu}{j}}{\binom{j + \nu}{j}} \frac{\binom{2 j + 2 \nu}{j}}{\binom{j +
    \nu}{j}} \frac{\binom{2 ( k - j) + 2 \nu}{k - j}}{\binom{k - j + \nu}{k -
    j}} . \label{eq:W4}
  \end{equation}
\end{lemma}

\begin{proof}
  We apply Corollary~\ref{cor:W:iter} with $n_1 = 2$ and $n_2 = 2$, to obtain
  \begin{equation*}
    W_4 ( \nu ; 2 k) = \sum_{j = 0}^k \frac{\binom{k}{j} \binom{k +
     \nu}{j}}{\binom{j + \nu}{j}} W_2 ( \nu ; 2 j) W_2 ( \nu ; 2 ( k - j)) .
  \end{equation*}
  Using the evaluation of $W_2 ( \nu ; 2 k)$ given by \eqref{eq:W2} then
  yields \eqref{eq:W4}.
\end{proof}

\begin{example}\dueto{Generalised Domb numbers}
  \label{eg:domb} The binomial sums in \eqref{eq:W4} generalize the {\emph{Domb
  numbers}}, also known as the \emph{diamond lattice} numbers \cite[\texttt{A002895}]{sloane-oeis},
  \begin{equation}
    W_4 ( 0 ; 2 k) = \sum_{j = 0}^k \binom{k}{j}^2 \binom{2 j}{j} \binom{2 ( k
    - j)}{k - j}, \label{eq:W4:d2}
  \end{equation}
  for $k=0,1,2,\ldots$, which have played an important role in dimension $2$. Their ordinary generating
  function,
  \begin{align}
    \sum_{k = 0}^{\infty} W_4 ( 0 ; 2 k) x^k &= \frac{1}{1 - 16 x} \pFq32{\frac{1}{3}, \frac{1}{2}, \frac{2}{3}}{1, 1}{\frac{108 x}{( 16 x - 1)^3}}
    \nonumber\\ & = \frac{1}{1 - 16 x} \pFq21{\frac{1}{6}, \frac{1}{3}}{1}{\frac{108 x}{( 16 x - 1)^3}}^2, \label{eq:W4:ogf:d2}
  \end{align}
  was determined in \cite{rogers-5f4}. The final equation follows from Clausen's product formula.

  In four dimensions, the recursive
  relation \eqref{eq:W:crec} combined with \eqref{eq:W2:d4} yields
  \begin{equation*}
    W_4 ( 1 ; 2 k) = \sum_{j = 1}^{k+1} N ( k + 1, j )\, C_{j } C_{k - j +
     2},
  \end{equation*}
  where $C_k$ are the Catalan numbers, as in \eqref{eq:W2:d4}, and
  \begin{equation*}
    N ( k + 1, j + 1) = \frac{1}{j + 1} \binom{k}{j} \binom{k + 1}{j}
  \end{equation*}
  are the Narayana numbers, as in Example \ref{eg:domb}.
  After developing some further properties of the
  moments, we illustrate in Example~\ref{eg:W4:ogf} that the ordinary generating
  function for the even moments $W_4 (\nu ; 2 k)$ can be expressed in terms of
  hypergeometric functions whenever the dimension is even.
\end{example}

\begin{corollary}\dueto{Hypergeometric form of $W_4$ at even integers}
  For $k=0,1,2,\ldots$, we have
  \begin{equation*}
    W_4 ( \nu; 2 k) = \frac{\binom{2 k + 2 \nu}{k}}{\binom{k + \nu}{k}} \pFq43{- k, - k - \nu, - k - 2 \nu, \nu + 1 / 2}{\nu + 1, 2 \nu + 1, - k - \nu + 1 / 2}{1} .
  \end{equation*}
\end{corollary}

We note that this hypergeometric function is \emph{well-poised} \cite[ \S16.4]{DLMF}.

\begin{example}
  Applying creative telescoping to the binomial sum \eqref{eq:W4}, we derive
  that the moments $W_4 ( \nu ; 2 k)$ satisfy the recursion
  \begin{eqnarray}
    &  & ( k + 2 \nu + 1) ( k + 3 \nu + 1) ( k + 4 \nu + 1) W_4 ( \nu ; 2 k +
    2) \nonumber\\
    & = & \left( \left( k + \tfrac{1}{2} \right) + 2 \nu \right) \left( 20
    \left( k + \tfrac{1}{2} \right)^2 + 80 \left( k + \tfrac{1}{2} \right) \nu
    + 48 \nu^2 + 3 \right) W_4 ( \nu ; 2 k) \nonumber\\
    &  & - 64 k ( k + \nu) ( k + 2 \nu) W_4 ( \nu ; 2 k - 2).  \label{eq:rec4}
  \end{eqnarray}
  As in Example~\ref{eg:rec3}, $W_4 ( \nu ; s)$ is analytic, exponentially
  bounded for $\Re s \geq 0$ and bounded on vertical lines. Hence,
  Carlson's Theorem~\ref{thm:carlson} applies to show that the recursion
  \eqref{eq:rec4} extends to complex $k$.
\end{example}

The following result is the counterpart of Theorem~\ref{thm:W3:meijer} and
extends \cite[Theorem 2.8]{bsw-rw2}.

\begin{theorem}\dueto{Meijer $G$ form of $W_4$}
  \label{thm:W4:meijer}For all complex $s$ with $\Re (s) > - 4 \nu - 2$
  and dimensions $d \geq 2$,
  \begin{equation*}
    W_4 ( \nu ; s) = 2^{s + 4 \nu} \nu !^3 \frac{\Gamma \left( \frac{s}{2} +
     \nu + 1 \right)}{\Gamma \left( \frac{1}{2} \right)^2 \Gamma \left( -
     \frac{s}{2} \right)} \MeijerG{2, 2}{4, 4}{1, \tfrac{1 - s}{2} - \nu, 1 + \nu, 1 + 2 \nu}{\tfrac{1}{2} + \nu, - \tfrac{s}{2}, - \tfrac{s}{2} - \nu, -
       \tfrac{s}{2} - 2 \nu}{1} .
  \end{equation*}
\end{theorem}

\begin{proof}
  The proof is obtained along the lines of the proof of
  Theorem~\ref{thm:W3:meijer}. This time, Parseval's formula is applied to the
  product $j_{\nu}^4 (x) = j_{\nu}^2 (x) j_{\nu}^2 (x)$.
\end{proof}

\begin{remark}
  We note that, as in the case of \eqref{eq:W3:meijer:hyp} for three steps,
  the Meijer $G$-function in Theorem~\ref{thm:W4:meijer} can be expressed as a sum of
  hypergeometric functions, namely
  \begin{eqnarray}
    W_4 (\nu ; s) & = & d_1 \cdot\pFq43{\tfrac{1}{2}, \tfrac{1}{2} - \nu, \tfrac{1}{2} + \nu, 2 \nu +
      \tfrac{s}{2} + 1}{\nu + \tfrac{s + 3}{2}, 2 \nu + \tfrac{s + 3}{2}, 3 \nu + \tfrac{s +
      3}{2}}{1} \nonumber\\
    &  & + d_2 \cdot\pFq43{- \tfrac{s}{2}, - \tfrac{s}{2} - \nu, - \tfrac{s}{2} - 2 \nu, \nu +
      \tfrac{1}{2}}{\nu + 1, \nu + 2, - \nu - \tfrac{s - 1}{2}}{1},  \label{eq:W4:meijer:hyp}
  \end{eqnarray}
  by Slater's Theorem \cite[p.~57]{mar}. The coefficients $d_1$ and $d_2$ of
  the two hypergeometric functions are
  \begin{equation*}
    d_1 \assign \frac{2^{4 \nu + s + 1} \Gamma (\nu + 1) \Gamma (2 \nu +
     \tfrac{s}{2} + 1)}{\sqrt{\pi} \Gamma (3 \nu + \tfrac{s + 3}{2})} c_1,
     \hspace{1em} d_2 \assign c_2 ,
  \end{equation*}
  where the coefficients $c_1$ and
  $c_2$ are as in the three-step case, see Remark~\ref{rk:W3:meijer:hyp}.
\end{remark}

\begin{example}\dueto{Odd moments $W_4(\nu;\cdot)$ in even dimensions}
  It was shown in \cite[Section 2.3 \& 3.1]{bsw-rw2} that the average
  distance to the origin after four random steps in the plane, as well as all its odd
  moments, can be evaluated in terms of the elliptic integrals
  \begin{eqnarray*}
    A & \assign & \frac{1}{\pi^3} \int_0^1 K' ( k)^2 \mathd k = \frac{\pi}{16} \pFq76{\tfrac{5}{4}, \tfrac{1}{2}, \tfrac{1}{2}, \tfrac{1}{2}, \tfrac{1}{2},
      \tfrac{1}{2}, \tfrac{1}{2}}{\tfrac{1}{4}, 1, 1, 1, 1, 1}{1},\\
    B & \assign & \frac{1}{\pi^3} \int_0^1 k^2 K' ( k)^2 \mathd k = \frac{3
    \pi}{256} \pFq76{\tfrac{7}{4}, \tfrac{3}{2}, \tfrac{3}{2}, \tfrac{1}{2}, \tfrac{1}{2},
      \tfrac{1}{2}, \tfrac{1}{2}}{\tfrac{3}{4}, 2, 2, 2, 2, 1}{1} .
  \end{eqnarray*}
  We note that each of the ${}_7 F_6$ hypergeometric functions may alternatively
  be expressed as the sum of two ${}_6 F_5$ hypergeometric functions. Then, in
  two dimensions,
  \begin{equation*}
    W_4 ( 0 ; - 1) = 4 A , \hspace{1em} W_4 ( 0 ; 1) = 16 A - 48 B,
  \end{equation*}
  and it follows from \eqref{eq:rec4}, extended to complex values, that all
  odd moments are indeed rational linear combinations of $A$ and $B$. In order
  to generalize this observation to higher dimensions, we claim that
  \begin{eqnarray}
    0 & = & 3 ( s + 2) ( s + 6 \nu) ( s + 8 \nu) ( s + 8 \nu - 2) W_4 ( \nu ;
    s) \nonumber\\
    &  & + 256 \nu^3 ( s + 4 \nu) W_4 ( \nu - 1 ; s + 2) - 8 \nu^3 ( 5 s + 32
    \nu - 6) W_4 ( \nu - 1 ; s + 4) .  \label{eq:W4:rec:nu:s}
  \end{eqnarray}
  This counterpart of \eqref{eq:W3:rec:nu:s} can be proved using the holonomic systems
  approach \cite{zeilberger90} applied to the hypergeometric form \eqref{eq:W4:meijer:hyp}. We
  conclude that, in any even dimension, the odd moments lie in the
  $\mathbb{Q}$-span of the constants $A$ and $B$, which arose in the planar
  case. For instance, we find that the average distance after four random
  steps in four dimensions is
  \begin{equation*}
    W_4 ( 1 ; 1) = \frac{3334144}{165375} A - \frac{11608064}{165375} B,
  \end{equation*}
  and so on.
\end{example}

\begin{example}\dueto{First derivative of $W_4(\nu;\cdot)$ in even dimensions}
  In continuation of Example~\ref{eg:W3:der}, we recall from \cite[Examples
  6.2 and 6.6]{bswz-densities} that
  \begin{equation*}
    W_4' (0 ; 0) = \frac{7 \zeta (3)}{2 \pi^2}, \hspace{1em} W_4' (0 ; 2) =
     \frac{14 \zeta (3)}{\pi^2} - \frac{12}{\pi^2} + 3,
  \end{equation*}
  where, again, the derivatives are with respect to $s$. These evaluations
  may, for instance, be obtained from differentiating the hypergeometric
  expression \eqref{eq:W4:meijer:hyp}. Proceeding as in
  Example~\ref{eg:W3:der}, we differentiate both \eqref{eq:rec4}, extended to
  complex values, and \eqref{eq:W4:rec:nu:s}, we conclude that, for all
  integers $k, \nu \geq 0$,
  \begin{equation*}
    W_4' (\nu ; 2 k) = r_1 + r_2 \frac{1}{\pi^2} + r_3 \frac{7 \zeta (3)}{2
     \pi^2},
  \end{equation*}
  with rational numbers $r_1, r_2, r_3$. Again, we find that $r_3 = 1$ in the
  case $k = 0$. In four and six dimensions, we obtain, for example,
  \begin{equation*}
    W_4' (1 ; 0) = \frac{3}{4} - \frac{53}{9 \pi^2} + \frac{7 \zeta (3)}{2
     \pi^2}, \hspace{1em} W_4' (2 ; 0) = \frac{13}{24} - \frac{48467}{14175
     \pi^2} + \frac{7 \zeta (3)}{2 \pi^2} .
  \end{equation*}
  In analogy with the case of three-step walks, the derivative $W_4' (0 ; 0)$
  is particularly interesting because it is the Mahler measure of the
  multivariate polynomial $1 + x_1 + x_2 + x_3$.
\end{example}

\begin{example}\dueto{Second derivative of $W_4(\nu;\cdot)$ in even dimensions}
  In \cite[(6.5), (6.10)]{bswz-densities} evaluations for the second
  derivatives $W_4'' (0 ; 0)$ and $W_4'' (0 ; 2)$ are given in terms of
  polylogarithmic constants. It follows from these evaluations and the
  functional equation \eqref{eq:rec4} that the derivative values $W_4'' (0 ; 2
  k)$ all lie in the $\mathbb{Q}$-linear span of
  \begin{equation*}
    1, \hspace{1em} \pi^2, \hspace{1em} \log^2 2, \hspace{1em}
     \frac{1}{\pi^2}, \hspace{1em} \frac{\log 2}{\pi^2}, \hspace{1em}
     \frac{\zeta (3)}{\pi^2}, \hspace{1em} \frac{\log^4 2}{\pi^2},
     \hspace{1em} \frac{\zeta (3) \log 2}{\pi^2}, \hspace{1em}
     \frac{\operatorname{Li}_4 (1 / 2)}{\pi^2},
  \end{equation*}
  where $\operatorname{Li}_n (z)\assign\sum_{k \ge 1} z^k/k^n$ is the polylogarithm of order $n$. Indeed, we
  realize from the dimensional recursion \eqref{eq:W4:rec:nu:s} that the same
  is true for the values $W_4'' (\nu ; 2 k)$ in all even dimensions. For
  example,
  \begin{align*}
    W_4'' (1 ; 0) & = \frac{253}{432} + \frac{5 W_4' (0 ; 0)}{162} -
    \frac{239 W_4' (0 ; 2)}{648} - \frac{26 W_4'' (0 ; 0)}{27} + \frac{53
    W_4'' (0 ; 2)}{108}\\
    & = - \frac{25}{48} - \frac{1}{5} \pi^2 - \log^2 2 + \frac{1193}{54
    \pi^2} - \frac{106}{3} \frac{\log 2}{\pi^2} + \frac{21}{4} \frac{\zeta
    (3)}{\pi^2} \\
    &\qquad+ \frac{\log^4 2}{\pi^2} + 21 \frac{\zeta (3) \log
    2}{\pi^2} + 24 \frac{\operatorname{Li}_4 (1 / 2)}{\pi^2} .
  \end{align*}
  The number of (presumed) transcendental constants can be somewhat reduced
  when working in terms of Kummer-type polylogarithms, as illustrated in
  \cite[Theorem~4.7]{logsin1}.
\end{example}

\begin{example}\dueto{Ordinary generating function for even moments with four steps}
  \label{eg:W4:ogf}In \eqref{eq:W4:ogf:d2} we noted that the ordinary
  generating function of the moments $W_4 (0 ; 2 k)$ has a concise
  hypergeometric expression. It is natural to wonder if this result
  extends to higher dimensions.

  Combining the recursive relations
  \eqref{eq:rec4} and \eqref{eq:W4:rec:nu:s}, we are able to derive that the
  ordinary generating function of the moments $W_4 (\nu ; 2 k)$, when
  complemented with an appropriate principal part (as in
  Theorem~\ref{thm:W3:gf}), can be obtained from the corresponding generating
  function of $W_4 (\nu - 1 ; 2 k)$ as well as its first two derivatives.
  Because the precise relationship is not too pleasant, we only record the
  simplified generating function,
  \begin{equation}\label{eq:d41}
    - \frac{1}{2 x^2} + \frac{1}{x} + \sum_{n = 0}^{\infty} W_4 (1 ; 2 k) x^k
     = (32 x - 7) F_0^2 - (4 x - 1) \left[ ( 32 x +3)F_0 F_1 - \left(16 x^2 + 10 x + \frac14\right)F_1^2 \right],
  \end{equation}
  that we obtain in dimension $4$. Here,
  \begin{equation*}
    F_{\lambda} \assign \frac{1}{2\cdot 3^\lambda x( 16 x - 1)^{1 - \lambda}}
     \frac{\mathd^{\lambda}}{\mathd x^{\lambda}} \pFq21{\tfrac{1}{6}, \tfrac{1}{3}}{1}{\frac{108 x}{( 16 x - 1)^3}},
  \end{equation*}
  which, by the differentiation formula \cite[(16.3.1)]{DLMF}
  \begin{equation}
    \frac{\mathd^n}{\mathd x^n} \pFq{p}{q}{a_1, \ldots, a_p}{b_1, \ldots, b_q}{x} = \frac{( a_1)_n \cdots ( a_p)_n}{( b_1)_n \cdots
    ( b_q)_n} \pFq{p}{q}{a_1 + 1, \ldots, a_p + 1}{b_1 + 1, \ldots, b_q + 1}{x}, \label{eq:hyp:D}
  \end{equation}
  can be expressed as generalized hypergeometric functions. It would be nice
  if it was possible to make the general case as explicit as we did in
  Theorem~\ref{thm:W3:gf} for three-step walks, but we have not succeeded in
  doing so.
\end{example}

\subsection{Moments of $5$-step walks}\label{ssec:mom5}

As in the planar case, as well as in many related problems, it is much harder
to obtain explicit results in the case of five or more steps. This is
reflected, for instance, in the fact that an application of
Corollary~\ref{cor:W:iter} with $n_1 = 3$ and $n_2 = 2$, and appealing to
\eqref{eq:W3}, results in a double (and not single) sum of hypergeometric
terms.

\begin{lemma}
  The nonnegative even moments for a $5$-step walk in $d$ dimensions are
  \begin{eqnarray}
    W_5 ( \nu ; 2 k) & = & \sum_{j = 0}^k \binom{k}{j} \frac{\binom{k +
    \nu}{j}}{\binom{j + \nu}{j}} \frac{\binom{2 ( k - j + \nu)}{k -
    j}}{\binom{k - j + \nu}{k - j}} W_3 (\nu ; 2 j) \nonumber\\
    & = & \sum_{j = 0}^k \binom{k}{j} \frac{\binom{k + \nu}{j}}{\binom{j +
    \nu}{j}} \frac{\binom{2 ( k - j + \nu)}{k - j}}{\binom{k - j + \nu}{k -
    j}} \sum_{i = 0}^j \binom{j}{i} \frac{\binom{j + \nu}{i} \binom{2 i + 2
    \nu}{i}}{\binom{i + \nu}{i}^2}.  \label{eq:W5}
  \end{eqnarray}
\end{lemma}

\begin{example}
  \label{eg:W5:rec}As in the case of three and four steps, we can apply
  creative telescoping to the binomial sum \eqref{eq:W5} to derive a recursion
  for the moments $W_5 ( \nu ; 2 k)$. In contrast to the three-term recursions
  \eqref{eq:rec3} and \eqref{eq:rec4}, we now obtain a four-term recursion,
  namely
  \begin{align}
   & ( k + 2 \nu + 1) ( k + 3 \nu + 1) ( k + 4 \nu + 1) ( k + 5 \nu + 1)
    W_5 ( \nu ; 2 k + 2) \label{eq:W5:rec}\\
  & = a \left( \nu ; k + \tfrac{1}{2} \right) W_5 ( \nu ; 2 k) - b ( \nu
    ; k) W_5 ( \nu ; 2 k - 2) + 225 k (k - 1) (k + \nu) (k - 1 + \nu) W_5 (
    \nu ; 2 k - 4),  \nonumber
  \end{align}
  where
  \begin{eqnarray*}
    a (\nu ; m) & \assign & 35 m^4 + 350 \nu m^3 + \left( 1183 \nu^2 + \tfrac{21}{2}
    \right) m^2\\
    &  & + \left( 1540 \nu^2 + \tfrac{105}{2} \right) \nu m + \left( 600
    \nu^4 + \tfrac{237}{4} \nu^2 + \tfrac{3}{16} \right),\\
    b (\nu ; k) & \assign & k (k + \nu) (259 k^2 + 1295 k \nu + 1450 \nu^2 + 26) .
  \end{eqnarray*}
  As in Example~\ref{eg:rec3}, $W_5 ( \nu ; s)$ is analytic and suitably
  bounded for $\Re s \geq 0$, so that we may conclude from
  Carlson's Theorem~\ref{thm:carlson} that the recursion \eqref{eq:W5:rec}
  extends to complex $k$.
\end{example}

\begin{example}\dueto{Dimensional recursion for $W_5(\nu;s)$}
  \label{eg:W5:rec:dim}Creative telescoping, applied to the binomial sum
  \eqref{eq:W5} for $W_5 (\nu ; 2 k)$, allows us to derive the following more
  involved counterpart of the dimensional recursions \eqref{eq:W3:rec:nu:s},
  \eqref{eq:W4:rec:nu:s} for three and four steps.
  \begin{eqnarray}
    0 & = & 3 (s + 2) (s + 4) (s + 2 \nu + 2) (s + 8 \nu) (s + 10 \nu - 2) (s
    + 10 \nu) W_5 (\nu ; s) \nonumber\\
    &  & - 450 \nu^4 (s + 4) (s + 2 \nu + 2) W_5 (\nu - 1 ; s + 2)
    \nonumber\\
    &  & + 4 \nu^4 a (\nu ; s) W_5 (\nu - 1 ; s + 4) - 2 \nu^4 b (\nu ; s)
    W_5 (\nu - 1 ; s + 6),  \label{eq:W5:rec:dim}
  \end{eqnarray}
  where
  \begin{eqnarray*}
    a (\nu ; s) & \assign & 107 s^2 + 2 ( 445 \nu + 152) s + 2 ( 550 \nu^2 + 1165
    \nu - 78),\\
    b (\nu ; s) & \assign & ( s + 4 \nu + 2) ( 13 s + 110 \nu - 16) .
  \end{eqnarray*}
  This recursion is first obtained for nonnegative even integers $s$, and then
  extended to complex values using Carlson's Theorem~\ref{thm:carlson}.
\end{example}

\section{Densities of short random walks}\label{sec:den}

\subsection{Densities of $2$-step walks}\label{ssec:den2}

We find an explicit formula for the probability density $p_2 (\nu ; x)$ of the
distance to the origin after two random steps in $\mathbb{R}^d$ by computing the
Bessel integral \eqref{eq:pn:bessel}. An equivalent formula is given in
\cite[Corollary~4.2]{wan-phd}, which exploited the fact that the probability
density is essentially the inverse Mellin transform of the moments which are
evaluated in Theorem~\ref{thm:W2}.

\begin{lemma}
  The probability density function of the distance to the origin in $d
  \geq 2$ dimensions after $2$ steps is, for $0 < x < 2$,
  \begin{equation}
    p_2 (\nu ; x) = \frac{2}{\pi \binom{2 \nu}{\nu}} x^{2 \nu} ( 4 - x^2)^{\nu
    - 1 / 2} . \label{eq:p2}
  \end{equation}
\end{lemma}

\begin{proof}
  It follows from \eqref{eq:pn:besselj} that
  \begin{equation*}
    p_2 (\nu ; x) = \nu !2^{\nu} x^{\nu + 1} \int_0^{\infty} t^{1 - \nu}
     J_{\nu} (t x) J_{\nu}^2 ( t) \mathd t.
  \end{equation*}
  From \cite[Chapter 13.46]{watson-bessel}, we have the integral evaluation
  \begin{equation}
    \int_0^{\infty} t^{1 - \nu} J_{\nu} (a t) J_{\nu} (b t) J_{\nu} (c t)
    \mathd t = \frac{2^{\nu - 1} \Delta^{2 \nu - 1}}{(a b c)^{\nu} \Gamma (\nu
    + 1 / 2) \Gamma (1 / 2)}, \label{eq:int:J3}
  \end{equation}
  assuming that $\Re (\nu) > - 1 / 2$ and that $a, b, c$ are the sides
  of a triangle of area $\Delta$. In our case,
  \begin{equation*}
    \Delta = \tfrac{x}{2} \sqrt{1 - ( \tfrac{x}{2})^2},
  \end{equation*}
  and therefore
  \begin{equation*}
    p_2 (\nu ; x) = \frac{\nu !}{\Gamma (\nu + 1 / 2) \Gamma (1 / 2)} x^{2
     \nu} \left( 1 - ( \tfrac{x}{2})^2 \right)^{\nu - 1 / 2},
  \end{equation*}
  which is equivalent to \eqref{eq:p2}.
\end{proof}

Note that \eqref{eq:p2} reflects the general fact, discussed in
Section~\ref{ssec:odd}, that the densities $p_n (\nu ; x)$ are piecewise
polynomial in odd dimensions.

For comparison with the case of $3$ steps, we record that the density $p_2 (
\nu ; x)$ satisfies the following functional equation: if $F ( x) \assign p_2 ( \nu
; x) / x$, then
\begin{equation}
  F ( x) = F \left( \sqrt{4 - x^2} \right) . \label{eq:p2:fun}
\end{equation}
\begin{remark}
  \label{rk:p2:fun}The probability density in two dimensions, that is
  \begin{equation*}
    p_2 (0 ; x) = \frac{2}{\pi \sqrt{4 - x^2}},
  \end{equation*}
  is readily identified as the distribution of $2 | \cos \theta |$ with
  $\theta$ uniformly distributed on $[0, 2 \pi]$. In other words, if $\boldsymbol{X} =
  (X_1, X_2)$ is uniformly distributed on the sphere of radius $2$ in
  $\mathbb{R}^2$, then $p_2 (0 ; x)$ describes the distribution of $|X_1 |$.
  It is then natural to wonder whether this observation extends to higher
  dimensions.

  Indeed, if $\boldsymbol{X} = (X_1, X_2)$ is uniformly distributed on the sphere of radius
  $2$ in $\mathbb{R}^{2 d - 2}$, where $\boldsymbol{X}$ is partitioned into $X_1, X_2 \in
  \mathbb{R}^{d - 1}$, then $p_2 (\nu ; x)$ describes the distribution of
  $|X_1 |$. Details are left to the interested reader. In this stochastic
  interpretation, the invariance of \eqref{eq:p2:fun} under $x \mapsto \sqrt{4
  - x^2}$ is a reflection of the fact that $|X_1 |$ and $|X_2 | = \sqrt{4 -
  |X_1 |^2}$ share the same distribution.
\end{remark}

\subsection{Densities of $3$-step walks}\label{ssec:den3}

It was shown in \cite{bswz-densities} that the density $p_3 (0 ; x)$ of the
distance to the origin after three random steps in the plane has the closed form
\begin{equation}
  p_3 ( 0 ; x) = \frac{2 \sqrt{3}}{\pi}  \frac{x}{3 + x^2} \pFq21{\tfrac{1}{3}, \tfrac{2}{3}}{1}{\frac{x^2 ( 9 - x^2)^2}{( 3 + x^2)^3}},
  \label{eq:p3:2d}
\end{equation}
valid on the interval $(0, 3)$. We next generalize this hypergeometric
expression to arbitrary dimensions. In order to do so, we need to establish
the behaviour of $p_3 ( \nu ; x)$ at the end points.

To begin with, we use information on the pole structure of the moments $W_3
(\nu ; s)$ to deduce the asymptotic behaviour of $p_3 ( \nu ; x)$ as $x
\rightarrow 0^+$.

\begin{proposition}
  \label{prop:p3:asy0} For positive half-integer $\nu$ and $x \rightarrow 0^+$,
  \begin{equation}
    p_3 ( \nu ; x) \sim \frac{2}{\sqrt{3} \pi}  \frac{3^{\nu}}{\binom{2
    \nu}{\nu}} x^{2 \nu + 1} . \label{eq:p3:asy0}
  \end{equation}
\end{proposition}

\begin{proof}
  As shown in Corollary~\ref{cor:W:res}, the moments $W_3 (\nu ; s)$ are
  analytic for $\Re s > - d$ and the first pole has residue
  \begin{equation*}
    \operatorname{Res}_{s = - d} W_3 (\nu ; s) = \nu !2^{\nu} \int_0^{\infty} x^{1 -
     \nu} J_{\nu}^3 ( x) \mathd x = \frac{2}{\sqrt{3} \pi}
     \frac{3^{\nu}}{\binom{2 \nu}{\nu}},
  \end{equation*}
  where the last equality is another special case of \eqref{eq:int:J3}. The
  asymptotic behaviour \eqref{eq:p3:asy0} then follows because $W_3 (\nu ; s -
  1)$ is the Mellin transform of $p_3 (\nu ; x)$.
\end{proof}

On the other hand, to obtain the asymptotic behaviour of $p_3 ( \nu ; x)$ as
$x \rightarrow 3^-$, we have to work a bit harder than in the case of the
behaviour \eqref{eq:p3:asy0} as $x \rightarrow 0^+$.

\begin{proposition}
  \label{prop:p3:asy3} For positive half-integer $\nu$ as $x \rightarrow 3^-$,
  \begin{equation}
    p_3 ( \nu ; x) \sim \frac{\sqrt{3}}{2 \pi}  \frac{2^{2 \nu}
    3^{\nu}}{\binom{2 \nu}{\nu}}  ( 3 - x)^{2 \nu} . \label{eq:p3:asy3}
  \end{equation}
\end{proposition}

\begin{proof}
  Using Theorem~\ref{thm:pn:rec} together with the fact that, for $x \in [0,
  2]$,
  \begin{equation*}
    p_2 (\nu ; x) = \frac{2}{\pi \binom{2 \nu}{\nu}} x^{2 \nu} ( 4 -
     x^2)^{\nu - 1 / 2},
  \end{equation*}
  we find
  \begin{equation*}
    p_3 (\nu ; x) = \frac{(2 x)^{2 \nu + 1}}{\binom{2 \nu}{\nu}^2 \pi^2}
     \int_{- 1}^{\min (1, \frac{3 - x^2}{2 x})} \frac{( 4 - (1 + 2 \lambda x +
     x^2))^{\nu - 1 / 2}}{(1 + 2 \lambda x + x^2)^{1 / 2}} (1 -
     \lambda^2)^{\nu - 1 / 2} \mathd \lambda .
  \end{equation*}
  Observe that the upper bound of integration is $1$ if $x \in [0, 1]$, and
  $(3 - x^2) / (2 x)$ for $x \in [1, 3]$. On $x \in [ 1, 3]$, after
  substituting $t = \sqrt{1 + 2 \lambda x + x^2} / x$, we therefore find
  \begin{equation*}
    p_3 (\nu ; x) = \frac{(2 x)^2}{\binom{2 \nu}{\nu}^2 \pi^2} \int_{1 -
     \frac{1}{x}}^{\frac{2}{x}} \{ ( t^2 x^2 - 4) ( t^2 x^2 - ( x + 1)^2) (
     t^2 x^2 - ( x - 1)^2) \}^{\nu - 1 / 2} \mathd t.
  \end{equation*}
  Note that the polynomial in the integrand factors into $(2 - t x) (t x - (x
  - 1))$ times a factor which approaches $192$ as $x \rightarrow 3^-$ and $t
  \rightarrow 2 / 3$. Hence, as $x \rightarrow 3^-$,
  \begin{equation*}
    p_3 (\nu ; x) \sim \frac{(2 x)^2 \cdot 192^{\nu - 1 / 2}}{\binom{2
     \nu}{\nu}^2 \pi^2} \int_{1 - \frac{1}{x}}^{\frac{2}{x}} \{ (2 - t x) (t x
     - (x - 1)) \}^{\nu - 1 / 2} \mathd t.
  \end{equation*}
  We now relate this integral to the incomplete beta function to find
  \begin{equation*}
    \int_{1 - \frac{1}{x}}^{\frac{2}{x}} \{ (2 - t x) (t x - (x - 1)) \}^{\nu
     - 1 / 2} \mathd t = \binom{2 \nu}{\nu} \frac{\pi (3 - x)^{2 \nu}}{2^{4
     \nu} x} .
  \end{equation*}
  Putting these together, we conclude that \eqref{eq:p3:asy3} holds.
\end{proof}

We are now in a position to generalize \eqref{eq:p3:2d} to higher dimensions.

\begin{theorem}\dueto{Hypergeometric form for $p_3$}
  For any half-integer $\nu \geq 0$ and $x \in [0, 3]$, we have
  \begin{equation}
    \frac{p_3 ( \nu ; x)}{x} = \frac{2 \sqrt{3}}{\pi}  \frac{3^{- 3
    \nu}}{\binom{2 \nu}{\nu}}  \frac{x^{2 \nu} ( 9 - x^2)^{2 \nu}}{3 + x^2} \pFq21{\tfrac{1}{3}, \tfrac{2}{3}}{1 + \nu}{\frac{x^2 ( 9 - x^2)^2}{( 3 + x^2)^3}} .
    \label{eq:p3:hyp}
  \end{equation}
\end{theorem}

\begin{proof}
  We observe that both sides of \eqref{eq:p3:hyp} satisfy the differential
  equation $A_3 \cdot y ( x) = 0$, where the differential operator $A_3$ is
  given by
  \begin{equation*}
    x ( x^2 - 9) ( x^2 - 1) D_x^2 - ( 2 \nu - 1) ( 5 x^4 - 30 x^2 + 9) D_x +
     4 x ( x^2 - 3) ( 3 \nu - 1) ( 2 \nu - 1),
  \end{equation*}
  where $D_x = \frac{\mathd}{\mathd x}$. Note that, in the case of $p_3 (\nu ;
  x)$, this is a consequence of the functional equation resulting from
  \eqref{eq:rec3}. The indices of this differential equation at $x = 0$ are
  $0$ and $2 \nu$. It follows from Proposition~\ref{prop:p3:asy0} that $p_3 (
  \nu ; x) / x$ is the unique solution $y ( x)$, on $( 0, 1)$, such that
  \begin{equation*}
    y ( x) \sim \frac{2}{\sqrt{3} \pi}  \frac{3^{\nu}}{\binom{2 \nu}{\nu}}
     x^{2 \nu}
  \end{equation*}
  as $x \rightarrow 0^+$. Since this property is satisfied by the right-hand
  side of \eqref{eq:p3:hyp} as well, it follows that \eqref{eq:p3:hyp} holds
  for all $x \in ( 0, 1)$.

  Similarly, to show \eqref{eq:p3:hyp} for all $x \in ( 1, 3)$, we use the
  fact that the differential equation has indices $0$ and $2 \nu$ at $x = 3$
  as well. In light of Proposition~\ref{prop:p3:asy3}, $p_3 ( \nu ; x) / x$ is
  the unique solution $y ( x)$, on $( 1, 3)$, such that, as $x \rightarrow
  3^-$,
  \begin{equation*}
    y ( x) \sim \frac{1}{2 \sqrt{3} \pi}  \frac{2^{2 \nu} 3^{\nu}}{\binom{2
     \nu}{\nu}}  ( 3 - x)^{2 \nu} .
  \end{equation*}
  Again, it is routine to verify that this property is also satisfied by the
  right-hand side of \eqref{eq:p3:hyp}. Hence, it follows that
  \eqref{eq:p3:hyp} holds for all $x \in ( 1, 3)$.
\end{proof}

As a consequence, we have the following functional equation for the
probability density function $p_3 ( \nu ; x)$. The role of the involution $x
\mapsto \sqrt{4 - x^2}$ for $2$ steps, see \eqref{eq:p2:fun}, is now played by
the involution $x \mapsto \frac{3 - x}{x + 1}$.

\begin{corollary}\dueto{Functional equation for $p_3$}\label{cor:fe3}
  For any half-integer $\nu \geq 0$ and $x \in [0, 3]$, the function $F ( x)
  \assign p_3 ( \nu ; x)/x$ satisfies the functional equation
  \begin{equation}
    F ( x) = \left( \frac{1 + x}{2} \right)^{6 \nu - 2} F \left( \frac{3 -
    x}{1 + x} \right) . \label{eq:p3:fun}
  \end{equation}
\end{corollary}

\begin{proof}
  The hypergeometric right-hand side of \eqref{eq:p3:hyp} clearly satisfies
  the functional equation \eqref{eq:p3:fun}.
\end{proof}

It would be very interesting to have a probabilistic interpretation of this
functional equation satisfied by the densities $p_3 ( \nu ; x)$.

\begin{remark}
  We note the following relation between the functional equations for the two-
  and three-step case in \eqref{eq:p2:fun} and \eqref{eq:p3:fun},
  respectively. Namely, if
  \begin{equation*}
    X = \frac{x + y}{2} + 1, \hspace{1em} Y = \frac{x - y}{2 i},
  \end{equation*}
  then the relation $y = \frac{3 - x}{1 + x}$ translates into $Y^2 = 4 - X^2$.
  It is natural to wonder whether this observation might help explain the
  functional equation \eqref{eq:p3:fun} for three-step densities.
\end{remark}

\begin{example}
  \label{eg:p3:dimrec}As a consequence of the contiguity relations satisfied
  by hypergeometric functions, we derive from \eqref{eq:p3:hyp} that the
  densities $p_3 ( \nu ; x)$ satisfy the dimensional recursion
  \begin{eqnarray*}
    p_3 ( \nu + 1 ; x) & = & \tfrac{\nu ( \nu + 1)^2}{6 ( 2 \nu + 1) ( 3 \nu +
    1) ( 3 \nu + 2)}  ( x^2 - 3) ( x^2 - 6 x - 3) ( x^2 + 6 x - 3) p_3 ( \nu ;
    x)\\
    &  & + \tfrac{\nu^2 ( \nu + 1)^2}{12 ( 2 \nu - 1) ( 2 \nu + 1) ( 3 \nu +
    1) ( 3 \nu + 2)} x^2 ( x^2 - 1)^2 ( x^2 - 9)^2 p_3 ( \nu - 1 ; x),
  \end{eqnarray*}
  where $\nu \geq 1$.
\end{example}

\begin{example}
  It follows from the hypergeometric formula \eqref{eq:p3:hyp} that the
  densities $p_3 ( \nu ; x)$, for $\nu > 0$, take the special values
  \begin{equation*}
    p_3 (\nu ; 1) = \frac{3}{4 \pi^2}  \frac{2^{6 \nu}}{\nu}  \frac{(\nu
     !)^5}{(2 \nu) ! (3 \nu) !} .
  \end{equation*}
  In particular, $p_3 (\nu ; 1) \in \mathbb{Q}$ in odd dimensions, and $p_3
  (\nu ; 1) \in \mathbb{Q} \cdot \frac{1}{\pi^2}$ in even dimensions.
\end{example}

\begin{example}\label{FE:diff}
  From \eqref{eq:pn:1:d} and the functional equation of Corollary~\ref{cor:fe3}, for $\nu>1$, we learn that
  \begin{equation*}
    p_3'''(\nu;1) = \frac92 \nu p_3''(\nu;1) - \frac38 (3\nu-1)(6\nu^2+5\nu+2) p_3(\nu;1),
  \end{equation*}
  but learn nothing about $p_3''(\nu;1)$.
\end{example}

\subsection{Densities of $4$-step walks}\label{ssec:den4}

It is shown in \cite{bswz-densities} that, in the planar case, the
probability density of the distance to the origin after four steps admits the
hypergeometric closed form
\begin{equation}
  p_4 ( 0 ; x) = \frac{2}{\pi^2}  \frac{\sqrt{16 - x^2}}{x} \Re \pFq32{\tfrac{1}{2}, \tfrac{1}{2}, \tfrac{1}{2}}{\tfrac{5}{6}, \tfrac{7}{6}}{\frac{( 16 - x^2)^3}{108 x^4}} . \label{eq:p4:2d}
\end{equation}
In this section we obtain a higher-dimensional analog of this evaluation by
demonstrating that hypergeometric formulae can be given for the densities $p_4
( \nu ; x)$ in all even dimensions.

\begin{example}
  Since the Mellin transform of the density $p_4 ( \nu ; x)$ is given by the
  corresponding probability moments $W_4 (\nu ; s - 1)$, the recursion
  \eqref{eq:rec4} for the moments translates into a differential equation for
  the density. We refer to \cite{bswz-densities} for details. We thus find
  that $p_4 ( \nu ; x) / x$ is annihilated by the differential operator
  \begin{eqnarray}
    &  & x^3 ( x^2 - 16) ( x^2 - 4) D_x^3 - 3 ( 2 \nu - 1) x^2 ( 3 x^4 - 40
    x^2 + 64) D_x^2 \nonumber\\
    &  & + ( 2 \nu - 1) x ( ( 52 \nu - 19) x^4 - ( 416 \nu - 152) x^2 + 256
    \nu - 64) D_x \nonumber\\
    &  & - 8 ( 4 \nu - 1) ( 3 \nu - 1) ( 2 \nu - 1) x^2 ( x^2 - 4),
    \label{eq:p4:de}
  \end{eqnarray}
  where $D_x = \frac{\mathd}{\mathd x}$. We conclude that, in the planar case,
  $p_4 (0 ; x)$ satisfies the differential equation
  \begin{eqnarray}
    0 & = & x^3 ( x^2 - 16) ( x^2 - 4) p_4''' (0 ; x) + 6 x^2 ( x^4 - 10 x^2)
    p_4'' (0 ; x) \nonumber\\
    &  & + x ( 7 x^4 - 32 x^2 + 64) p_4' (0 ; x) + ( x^4 - 64) p_4 (0 ; x),
    \label{eq:p4:2d:de}
  \end{eqnarray}
  which agrees (up to a typo there) with \cite[(2.7)]{bswz-densities}.

  The differential equation \eqref{eq:p4:2d:de} is the symmetric square of a
  second order differential equation and, moreover, admits modular
  parametrization \cite[Remark~4.11]{bswz-densities}. These ingredients
  ultimately lead to the hypergeometric closed form \eqref{eq:p4:2d}. The
  differential equation associated with \eqref{eq:p4:de}, on the other hand,
  is a symmetric square of a second order differential equation only in the
  cases $\nu = 0$ and $\nu = 1 / 2$.
\end{example}

\begin{example}
  In three dimensions, that is when $\nu = 1 / 2$, the density is
  \begin{equation}
    p_4 (1 / 2 ; x) = \frac{x}{16}  (8 - x^2 + 2 (x - 2) |x - 2|),
    \label{eq:p4:3d:2}
  \end{equation}
  for $x \in [0, 4]$. This is equivalent to \eqref{eq:p4:3d} and may be derived
  directly from \eqref{eq:pn:bessel} or from Theorem~\ref{thm:pn:odd}.
\end{example}

\begin{example}
  Basic Mellin calculus connects the asymptotic behaviour of $p_n ( \nu ; x)$
  as $x \rightarrow 0^+$ with the nature of the poles of $W_n ( \nu ; s)$ in
  the left half-plane. For instance, from the explicit information in
  Example~\ref{eg:W4:poles} on the poles of $W_4 ( 1 ; s)$ at $s = - 4$ and $s
  = - 6$, we conclude that
  \begin{align*}
    p_4 ( 1 ; x) = \frac{4}{\pi^2} x^3 + \left( \frac{1}{16 \pi^2} -
    \frac{9 \log ( 2)}{4 \pi^2} \right) x^5 + \frac{3}{4 \pi^2} \log ( x) x^5
    + O ( x^7)
  \end{align*}
  as $x \rightarrow 0^+$.
\end{example}

The following result connects the $4$-step density in $d$ dimensions with the
corresponding density in $d - 2$ dimensions. In particular, using
\eqref{eq:p4:2d} and \eqref{eq:p4:3d:2} as base cases, this provides a way to
obtain explicit formulas for the densities $p_4 (\nu ; x)$ in all dimensions.

\begin{theorem}\dueto{Dimensional recursion for $p_4$}
  \label{thm:p4:rec}For $0 \le x \le 4$ and any half-integer $\nu \geq 0$,
  \begin{eqnarray}\label{eq:p4nu}
    &  & \frac{3 (2 \nu + 1) (3 \nu + 1) (3 \nu + 2) (4 \nu + 1) (4 \nu +
    3)}{64 (\nu + 1)^3} p_4 (\nu + 1 ; x)\\ \nonumber
    & = & - a (\nu ; x^2 / 8) p_4 (\nu ; x) + x b (\nu ; x^2 / 8) p_4' (\nu ;
    x) - c (\nu ; x^2 / 8) p_4'' (\nu ; x),
  \end{eqnarray}
  where
  \begin{eqnarray*}
    a (\nu ; x) & \assign & (4 \nu - 1) (6 \nu - 1) x^4 + 2 (100 \nu^2 - 23 \nu - 1)
    x^3 + 2 (2 \nu + 3) (12 \nu + 1) x^2\\
    &  & - 2 (60 \nu^2 + 13 \nu + 1) x + (2 \nu + 1) (4 \nu + 1),\\
    b (\nu ; x) & \assign & (10 \nu - 3) x^4 + 5 (12 \nu - 1) x^3 - \tfrac{21}{2} (8
    \nu - 1) x^2 - 20 \nu x + 6 \nu + 1,\\
    c (\nu ; x) & \assign & 4 x (x - 2) (2 x - 1) (x^2 + 5 x + 1) .
  \end{eqnarray*}
\end{theorem}

\begin{proof}
  We have already observed in \eqref{eq:rec4} and \eqref{eq:W4:rec:nu:s} that
  the moments $W_4 (\nu ; s)$ satisfy a functional equation connecting $W_4
  (\nu ; s)$, $W_4 (\nu ; s + 2)$, $W_4 (\nu ; s + 4)$, as well as a
  functional equation relating $W_4 (\nu + 1 ; s)$, $W_4 (\nu ; s + 2)$, $W_4
  (\nu ; s + 4)$. The usual Mellin calculus translates the first of these two
  into the third-order differential equation recorded in \eqref{eq:p4:de},
  while the second is translated into a more complicated equation involving
  seven terms ranging from $p_4 (\nu ; x)$ to $p_4^{(4)} (\nu + 1 ; x)$. With
  assistance of the \emph{Mathematica} package \texttt{Holo\-nomic\-Functions}, which
  accompanies Koutschan's thesis \cite{koutschan-phd}, we compute a
  Gr\"obner basis for the ideal that these two relations generate and use it
  to find the claimed relation involving $p_4 (\nu + 1 ; x)$ as well as $p_4
  (\nu ; x)$ and its first two derivatives.
\end{proof}

\begin{example}\dueto{Hypergeometric form for $p_4(1;x)$}\label{eg:p41}
  By combining \eqref{eq:p4:2d} with
  Theorem~\ref{thm:p4:rec} and \eqref{eq:hyp:D}, we conclude, for instance,
  that the $4$-step density in four dimensions, for $x \in (2, 4)$, can be
  hypergeometrically represented as
  \begin{equation}
    p_4 ( 1 ; x) = \frac{( 16 - x^2)^{5 / 2}}{( 24 \pi)^2 x} \left[ - ( x^2 +
    8)^2 G_0 + \frac{r ( x^2)}{7!!\, x^4} G_1 + \frac{( 16 - x^2)^3 s (
    x^2)}{13!!\, ( 2 / 3)^4 x^8} G_2 \right], \label{eq:p41}
  \end{equation}
  with $( 2 n + 1) !! = ( 2 n + 1) ( 2 n - 1) \cdots 3 \cdot 1$ and
  \begin{equation*}
    G_{\lambda} \assign \pFq32{\tfrac{1}{2} + \lambda, \tfrac{1}{2} + \lambda, \tfrac{1}{2} +
       \lambda}{\tfrac{5}{6} + \lambda, \tfrac{7}{6} + \lambda}{\frac{( 16 - x^2)^3}{108 x^4}}
  \end{equation*}
  as well as
  \begin{align*}
    r ( x) & \assign x^5 + 55 x^4 + 1456 x^3 + 25664 x^2 - 90112 x - 262144,\\
    s ( x) & \assign ( x - 4) ( x + 32)^2 ( x^2 + 40 x + 64) .
  \end{align*}
  As in \eqref{eq:p4:2d}, taking the real part of the hypergeometric functions
  provides a formula for $p_4 ( 1 ; x)$ which is valid for $x \in (0, 2)$ as
  well. The above hypergeometric formula also provides a modular
  parametrization (in a suitably generalized sense) for $p_4 ( 1 ; x)$. This
  is a consequence of
  \begin{equation*}
    p_4 ( 0 ; 8 i x ( \tau)) = \frac{6 ( 2 \tau + 1)}{\pi} f ( \tau),
  \end{equation*}
  which is proved in \cite[(4.16)]{bswz-densities} and which involves the
  modular function
  \begin{equation*}
    x ( \tau) = \frac{\eta ( 2 \tau)^3 \eta ( 6 \tau)^3}{\eta ( \tau)^3 \eta
     ( 3 \tau)^3}
  \end{equation*}
  and the weight $2$ modular form
  \begin{equation*}
    f ( \tau) = \eta ( \tau) \eta ( 2 \tau) \eta ( 3 \tau) \eta ( 6 \tau) .
  \end{equation*}
  Differentiating this modular parametrization of $p_4 ( 0 ; x)$, we find that
  $p_4 ( 1 ; 8 i x ( \tau))$ can be expressed in terms of modular quantities
  such as $f' ( \tau) / x' ( \tau)$. Then inductively we obtain the like result for higher even dimensions.
\end{example}

\begin{example}\dueto{$p_4(\nu;2)$ in even dimensions}
  Motivated by \cite[Corollary 4.8]{bswz-densities}, which proves that
  \begin{equation*}
    p_4 ( 0 ; 2) = \frac{2^{7 / 3} \pi}{3 \sqrt{3}} \Gamma \left( \frac{2}{3}
     \right)^{- 6} = \frac{\sqrt{3}}{\pi} W_3 ( 0 ; - 1),
  \end{equation*}
  or $\pi / \sqrt{3} p_4 (0 ; 2) = W_3 (0 ; - 1)$, we discover that, for
  integers $\nu \geq 0$, $\pi / \sqrt{3} p_4 (\nu ; 2)$ is a rational
  combination of the moments $W_3 (0 ; - 1)$ and $W_3 (0 ; 1)$. For instance,
  \begin{equation*}
    \frac{\pi}{\sqrt{3}} p_4 ( 1 ; 2) = - \frac{14}{3} W_3 ( 0 ; - 1) +
     \frac{10}{3} W_3 ( 0 ; 1)
  \end{equation*}
  and
  \begin{equation*}
    \frac{\pi}{\sqrt{3}} p_4 ( 2 ; 2) = \frac{6656}{315} W_3 ( 0 ; - 1) -
     \frac{704}{63} W_3 ( 0 ; 1) .
  \end{equation*}
  To deduce the first of these from equation \eqref{eq:p4nu} takes a little care as $p_4(1;x)$ is not differentiable at $2$ and one must take the limit from the left in \eqref{eq:p41}; likewise for higher dimensions (note that $c(\nu;2)=0$).
  Theorem~\ref{thm:p4:rec} specializes to
  \begin{eqnarray*}
    &  & (2 \nu + 1) (3 \nu + 1) (3 \nu + 2) (4 \nu + 1) (4 \nu + 3) p_4 (\nu
    + 1 ; 2)\\
    & = & 4 (\nu + 1)^3 [(6 \nu + 1) (12 \nu - 7) p_4 (\nu ; 2) - 30 (6 \nu -
    1) p_4' (\nu ; 2)] .
  \end{eqnarray*}
  On the other hand, specializing a corresponding relation among $p_4 (\nu ;
  x)$, $p_4 (\nu + 1 ; x)$, $p_4 (\nu + 2 ; x)$ and $p_4' (\nu ; x)$, which
  one obtains as in the proof of Theorem~\ref{thm:p4:rec}, we find
  \begin{eqnarray*}
    &  & (2 \nu - 1) (2 \nu + 1) (3 \nu - 2) (3 \nu - 1) (3 \nu + 1) (3 \nu +
    2) (4 \nu + 1) (4 \nu + 3) p_4 (\nu + 1 ; 2)\\
    & = & - 16 (\nu + 1)^3 (2 \nu - 1) (3 \nu - 2) (3 \nu - 1) (54 \nu^2 + 1)
    p_4 (\nu ; 2)\\
    &  & + 576 \nu^3 (\nu + 1)^3 (6 \nu - 5) (6 \nu - 1) p_4 (\nu - 1 ; 2) .
  \end{eqnarray*}
\end{example}

\subsection{Densities of $5$-step walks}\label{ssec:den5}

The $5$-step densities for dimensions up to $9$ are plotted in
Figure~\ref{fig:p5}.

\begin{figure}[h]
  \centering
  \includegraphics[width=9cm]{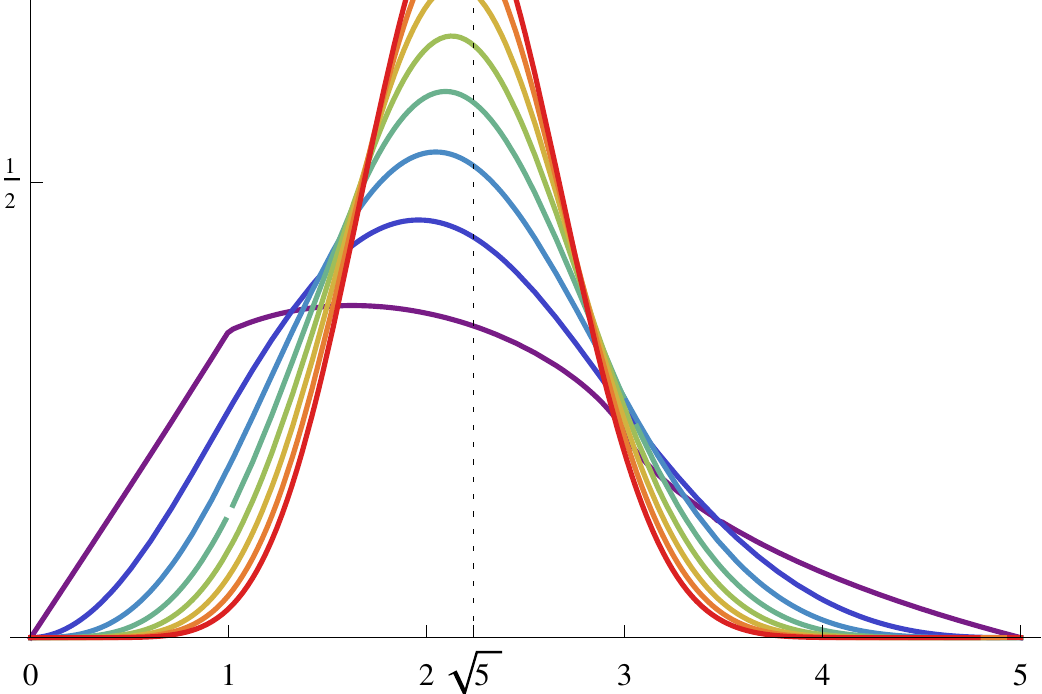}
  \caption{\label{fig:p5}$p_5 (\nu ; x)$ for $\nu = 0, \tfrac{1}{2}, 1,
  \ldots, \frac{7}{2}$}
\end{figure}

A peculiar feature of the planar density is its striking (approximate)
linearity on the initial interval $[0, 1]$. This phenomenon was already
observed by Pearson \cite{Pea06}, who commented on $p_5 (0 ; x) / x$,
between $x = 0$ and $x = 1$, as follows:

\begin{quote}
  ``the graphical construction, however carefully reinvestigated, did not
  permit of our considering the curve to be anything but a {\emph{straight}}
  line$\ldots$ Even if it is not absolutely true, it exemplifies the
  extraordinary power of such integrals of $J$ products [that is,
  \eqref{eq:pn:bessel}] to give extremely close approximations to such simple
  forms as horizontal lines.''
\end{quote}

Pearson's comment was revisited by Fettis \cite{fettis63}, who rigorously
established the nonlinearity. In \cite[Theorem~5.2]{bswz-densities}, it is
shown that the density satisfies a certain fourth-order differential equation,
and that, for small $x > 0$,
\begin{equation}
  p_5 (0 ; x) = \sum_{k = 0}^{\infty} r_{5, k} \hspace{0.25em} x^{2 k + 1},
  \label{eq:p5:2d:0}
\end{equation}
where
\begin{eqnarray}
  &  & ( 15 (2 k + 2) (2 k + 4))^2 r_{5, k + 2} \nonumber\\
  & = & ( 259 (2 k + 2)^4 + 104 (2 k + 2)^2) r_{5, k + 1} \nonumber\\
  &  & - ( 35 (2 k + 1)^4 + 42 (2 k + 1)^2 + 3) r_{5, k} + (2 k)^4 r_{5, k -
  1},  \label{eq:r5:2d:rec}
\end{eqnarray}
with initial conditions $r_{5, - 1} = 0$ and
\begin{equation*}
  r_{5, k} = \operatorname{Res}_{s = - 2 k - 2} W_5 (0 ; s) .
\end{equation*}
Numerically, we thus find that
\begin{equation*}
  p_5 (0 ; x) = 0.329934 \hspace{0.25em} x + 0.00661673 \hspace{0.25em} x^3 +
   0.000262333 \hspace{0.25em} x^5 + O (x^7),
\end{equation*}
which reflects the approximate linearity of $p_5 (0 ; x)$ for small $x$.

By \eqref{eq:pn:diffrel0}, the residue $r_{5, 0} = p_5' (0 ; 0)$ can be
expressed as $p_4 (0 ; 1)$. The modularity of $p_4$ in the planar case,
combined with the \emph{Chowla--Selberg formula} \cite{sc67}, then permits us to obtain the
explicit formula \cite[Theorem~5.1]{bswz-densities}
\begin{equation}
  r_{5, 0} = \frac{\sqrt{5}}{40}  \hspace{0.25em} \frac{\Gamma ( \frac{1}{15})
  \Gamma ( \frac{2}{15}) \Gamma ( \frac{4}{15}) \Gamma ( \frac{8}{15})}{\pi^4}
  . \label{eq:r50}
\end{equation}
Moreover, high-precision numerical calculations lead to the conjectural
evaluation \cite[(5.3)]{bswz-densities}
\begin{equation}
  r_{5, 1} = \frac{13}{225} r_{5, 0} - \frac{2}{5 \pi^4}  \frac{1}{r_{5, 0}},
  \label{eq:r51}
\end{equation}
and the recursion \eqref{eq:r5:2d:rec} then implies that all the coefficients
$r_{5, k}$ in \eqref{eq:p5:2d:0} can be expressed in terms of $r_{5, 0}$.

\begin{theorem}\dueto{Residues of $W_5(0;s)$}\label{thm:r51}
  The conjectural relation \eqref{eq:r51} is true.
\end{theorem}

\begin{proof}
  As noted after Corollary~\ref{cor:W:res}, we have
  \begin{equation*}
    \operatorname{Res}_{s = - 4} W_5 (1 ; s) = p_4 (1 ; 1) .
  \end{equation*}
  On the other hand, applying the dimensional recursion \eqref{eq:W5:rec:dim}
  for $W_5 (\nu ; s)$ with $\nu = 1$ and using the values $W_5 (0 ; 0) = 1$
  and $W_5 (0 ; 2) = 5$, we obtain
  \begin{equation*}
    \operatorname{Res}_{s = - 4} W_5 (1 ; s) = \frac{7}{6} - \frac{25}{32} r_{5, 0} +
     \frac{35}{24} W_5' (0 ; 0) - \frac{7}{24} W_5' (0 ; 2),
  \end{equation*}
  since $W_5 (0 ; s)$ has simple poles only \cite[Example~2.5]{bsw-rw2}. As
  noted in \cite[(6.2)]{bswz-densities}, it is a consequence of the
  functional equation \eqref{eq:W5:rec} that
  \begin{equation*}
    225 r_{5, 1} = 26 r_{5, 0} - 16 - 20 W_5' (0 ; 0) + 4 W_5' (0 ; 2) .
  \end{equation*}
  Combining the last three equations, we arrive at
  \begin{equation*}
    p_4 (1 ; 1) = \frac{107}{96} r_{5, 0} - \frac{525}{32} r_{5, 1} .
  \end{equation*}
  Equation \eqref{eq:r51} is therefore equivalent to
  \begin{equation}
    p_4 ( 1 ; 1) = \frac{1}{6} r_{5, 0} + \frac{105}{16 \pi^4}  \frac{1}{r_{5,
    0}} . \label{eq:p4:11}
  \end{equation}
  The equality \eqref{eq:p4:11} can now be deduced (at least in principle)
  from the hypergeometric formula for $p_4 ( 1 ; x)$, made explicit in
  Example~\ref{eg:p41}, the modular parametrization of $p_4 ( 0 ; x)$ as well
  as the Chowla--Selberg formula \cite{sc67}.
\end{proof}

In conclusion, we know that $p_5 (0 ; x)$ has a Taylor expansion
\eqref{eq:p5:2d:0} at $x = 0$, which converges and gives its values in the
interval $[0, 1]$. Moreover, we have a recursive description of the Taylor
coefficients and know that they are all $\mathbb{Q}$-linear combinations of
$r_{5, 0}$ in \eqref{eq:r50} and $1 / (\pi^4 r_{5, 0})$. All of these
statements carry over to the 5-step densities $p_5 (\nu ; x)$ in any even
dimension. Since the details are unwieldy, we only sketch why this is so.

Recall that the moments $W_5 (\nu ; s)$ satisfy the recursive relations
\eqref{eq:W5:rec} and \eqref{eq:W5:rec:dim}. Indeed, there is a third relation
which connects $W_5 (\nu ; s)$, $W_5 (\nu ; s + 2)$, $W_5 (\nu + 1 ; s)$, $W_5
(\nu + 1 ; s + 2)$. As in the proof of Theorem~\ref{thm:p4:rec}, the Mellin
transform translates these three recursive relations into (complicated)
differential relations for the densities $p_5 (\nu ; x)$. Assisted, once more,
by Koutschan's package \texttt{Holo\-nomic\-Functions}
\cite{koutschan-phd}, we compute a Gr\"obner basis for the ideal that
these three differential relations generate. From there, we find that there
exists, in analogy with Theorem~\ref{thm:p4:rec}, a relation
\begin{equation*}
  x^2 p_5 (\nu + 1 ; x) = A p_5 (\nu ; x) + B p_5' (\nu ; x) + C p_5'' (\nu ;
   x) + D p_5''' (\nu ; x),
\end{equation*}
with $A, B, C, D$ polynomials of degrees $12, 13, 14, 15$ in $x$ (with
coefficients that are rational functions in $\nu$).

We therefore conclude
inductively that, for integers $\nu$, the density $p_5 (\nu ; x)$ has a Taylor
expansion \eqref{eq:p5:2d:0} at $x = 0$ whose Taylor coefficients lie in the
$\mathbb{Q}$-span of $r_{5, 0}$ in \eqref{eq:r50} and $1 / (\pi^4 r_{5, 0})$.
It remains an open challenge, including in the planar case, to obtain a more
explicit description of $p_5 (\nu ; x)$.

\section{Conclusion}\label{sec:conc}

We have shown that quite delicate results are possible for densities and moments of walks in arbitrary dimensions, especially for two, three and four steps. We find it interesting that induction between dimensions provided methods to show Theorem~\ref{thm:r51}, a result in the plane that we previously could not establish \cite{bswz-densities}. We also should emphasize the crucial role played by computer experimentation and by computer algebra. One stumbling block is that currently \emph{Mathematica}, and to a lesser degree \emph{Maple}, struggle with computing various of the Bessel integrals to more than a few digits --- thus requiring considerable extra computational effort or ingenuity.

We leave some open questions:
\begin{itemize}
  \item The even moments $W_n ( 0 ; 2 k)$ associated to a random walk in two
  dimensions have combinatorial significance. They count {\emph{abelian
  squares}} \cite{richmond-absq09} of length $2 k$ over an alphabet with $n$
  letters (i.e., strings $x x'$ of length $2 k$ from an alphabet with $n$
  letters such that $x'$ is a permutation of $x$). As observed in
  Example~\ref{eg:W:values4d}, the even moments $W_n ( 1 ; 2 k)$ are positive
  integers as well and we have expressed them in terms of powers of the
  Narayana triangular matrix, whose entries count certain lattice paths. Does
  that give rise to an interpretation of the even four-dimensional moments
  themselves counting similar combinatorially interesting objects?

  \item As discussed in Example~\ref{eg:domb}, in the case $\nu = 0$, the
  moments $W_4 ( \nu ; 2 k)$ are the Domb numbers, for which a clean
  hypergeometric generating function is known. Referring to
  Example~\ref{eg:W4:ogf}, we wonder if it is possible to give a compact
  explicit hypergeometric expression for the generating function of the even
  moments $W_4 (\nu ; 2 k)$, valid for all even dimensions, as we did in
  Theorem~\ref{thm:W3:gf} for three-step walks.

  \item Verrill has exhibited \cite{verrill} an explicit recursion in $k$
  of the even moments $W_n ( 0 ; 2 k)$ in the plane. Combined with a result of
  Djakov and Mityagin \cite{dm1}, proved more directly and
  combinatorially by Zagier \cite[Appendix~A]{bswz-densities}, these
  recursions yielded insight into the general structure of the densities $p_n
  (0 ; x)$. For instance, as shown in \cite[Theorem~2.4]{bswz-densities}, it
  follows that these densities are real analytic except at $0$ and the
  positive integers $n, n - 2, n - 4, \ldots$ It would be interesting to
  obtain similar results for any dimension.

  \item By exhibiting recursions relating different dimensions, we have shown
  that the odd moments of the distances after three and four random steps in
  any dimension can all be expressed in terms of the constants arising in the
  planar case. Is it possible to evaluate these odd moments in a closed
  (hypergeometric) form which reflects this observation?

  \item In the plane, various other fragmentary modular results are
  (conjecturally) known for five and six step walks, see \cite[(6.11),
  (6.12)]{bswz-densities} for representations of $W'_5 (0 ; 0)$ and $W'_6 (0
  ; 0)$, conjectured by Rodriguez-Villegas \cite{banff-mahler}, as well
  as a discussion of their relation to {\emph{Mahler measures}}. Are more
  comprehensive results possible?
\end{itemize}

\begin{acknowledgements}
The third author would like to thank J.~M. Borwein for his invitation to visit
the CARMA center in September 2014; this research was initiated at this
occasion. Thanks are due to Ghislain McKay and Corey Sinnamon who explored the
even moments during visiting student fellowships at CARMA in early 2015.
\end{acknowledgements}


\begin{thebibliography}{BLRVD03}

\bibitem[AMV13]{amv-narayana}
T.~Amdeberhan, V.~H. Moll, and C.~Vignat.
\newblock A probabilistic interpretation of a sequence related to {N}arayana
  polynomials.
\newblock {\em Online Journal of Analytic Combinatorics}, 8, June 2013.

\bibitem[BBBG08]{b3g}
D.~H. Bailey, J.~M. Borwein, D.~M. Broadhurst, and L.~Glasser.
\newblock Elliptic integral representation of {B}essel moments.
\newblock {\em J. Phys. A: Math. Theory}, 41:5203--5231, 2008.

\bibitem[Ber13]{bernardi-rw}
O.~Bernardi.
\newblock A short proof of {R}ayleigh's {T}heorem with extensions.
\newblock {\em The American Mathematical Monthly}, 120(4):362--364, 2013.

\bibitem[BB01]{dbjb}
D.~Borwein and J.~M. Borwein.
\newblock Some remarkable properties of sinc and related integrals.
\newblock {\em Ramanujan J.}, 5:73--90, 2001.

\bibitem[BNSW11]{bnsw-rw}
J.~M. Borwein, D.~Nuyens, A.~Straub, and J.~Wan.
\newblock Some arithmetic properties of short random walk integrals.
\newblock {\em The Ramanujan Journal}, 26(1):109--132, 2011.

\bibitem[BS12]{logsin1}
J.~M. Borwein and A.~Straub.
\newblock Log-sine evaluations of {M}ahler measures.
\newblock {\em J. Aust Math. Soc.}, 92(1):15--36, 2012.

\bibitem[BSW13]{bsw-rw2}
J.~M. Borwein, A.~Straub, and J.~Wan.
\newblock Three-step and four-step random walk integrals.
\newblock {\em Experimental Mathematics}, 22(1):1--14, 2013.

\bibitem[BSWZ12]{bswz-densities}
J.~M. Borwein, A.~Straub, J.~Wan, and W.~Zudilin.
\newblock Densities of short uniform random walks (with an appendix by {D}on
  {Z}agier).
\newblock {\em Canadian Journal of Mathematics}, 64(5):961--990, 2012.

\bibitem[BLRVD03]{banff-mahler}
D.~Boyd, D.~Lind, F.~Rodriguez~Villegas, and C.~Deninger.
\newblock The many aspects of {M}ahler's measure.
\newblock Final report of 2003 Banff workshop. Available at
  \url{http://www.birs.ca/workshops/2003/03w5035/report03w5035.pdf}, 2003.

\bibitem[Bro09]{broadhurst-rw}
D.~J. Broadhurst.
\newblock {B}essel moments, random walks and {C}alabi-{Y}au equations.
\newblock Preprint, 2009.

\bibitem[DM04]{dm1}
P.~Djakov and B.~Mityagin.
\newblock Asymptotics of instability zones of {H}ill operators with a two term
  potential.
\newblock {\em C. R. Math. Acad. Sci. Paris}, 339(5):351--354, 2004.

\bibitem[{\relax DLMF}]{DLMF}
{NIST Digital Library of Mathematical Functions}.
\newblock http://dlmf.nist.gov/, Release 1.0.9 of 2014-08-29.
\newblock Online companion to \cite{NIST}.

\bibitem[Fet63]{fettis63}
H.~E. Fettis.
\newblock On a conjecture of {K}arl {P}earson.
\newblock {\em Rider Anniversary Volume}, pages 39--54, 1963.

\bibitem[GP12]{garcia-odd}
R.~Garc{\'i}a-Pelayo.
\newblock Exact solutions for isotropic random flights in odd dimensions.
\newblock {\em Journal of Mathematical Physics}, 53(10):103504, 2012.

\bibitem[Hug95]{hughes-rw}
B.~D. Hughes.
\newblock {\em Random Walks and Random Environments}, volume~1.
\newblock Oxford University Press, 1995.

\bibitem[Kin63]{kingman-rw}
J.~F.~C. Kingman.
\newblock Random walks with spherical symmetry.
\newblock {\em Acta Mathematica}, 109(1):11--53, July 1963.

\bibitem[Klu06]{Klu06}
J.~C. Kluyver.
\newblock A local probability problem.
\newblock {\em Nederl. Acad. Wetensch. Proc.}, 8:341--350, 1906.

\bibitem[Kou09]{koutschan-phd}
C.~Koutschan.
\newblock {\em Advanced Applications of the Holonomic Systems Approach}.
\newblock PhD thesis, RISC, Johannes Kepler University, Linz, Austria,
  September 2009.

\bibitem[Mar83]{mar}
O.~I. Marichev.
\newblock {\em Handbook of Integral Transforms of Higher Transcendental
  Functions: Theory and Algorithmic Tables}.
\newblock Ellis Horwood Limited, Chichester, England, 1983.

\bibitem[OLBC10]{NIST}
F.~W.~J. Olver, D.~W. Lozier, R.~F. Boisvert, and C.~W. Clark.
\newblock {\em NIST Handbook of Mathematical Functions}.
\newblock Cambridge University Press, New York, NY, 2010.
\newblock Print companion to \cite{DLMF}.

\bibitem[Pea05]{Pea05}
K.~Pearson.
\newblock The problem of the random walk.
\newblock {\em Nature}, 72(1866):294, 1905.

\bibitem[Pea06]{Pea06}
K.~Pearson.
\newblock A mathematical theory of random migration.
\newblock In {\em Drapers Company Research Memoirs}, number~3 in Biometric
  Series. Cambridge University Press, 1906.

\bibitem[Ray19]{rayleigh-flights}
L.~Rayleigh.
\newblock On the problem of random vibrations, and of random flights in one,
  two, or three dimensions.
\newblock {\em Philosophical Magazine Series 6}, 37(220):321--347, April 1919.

\bibitem[RS09]{richmond-absq09}
L.~B. Richmond and J.~Shallit.
\newblock Counting abelian squares.
\newblock {\em The Electronic Journal of Combinatorics}, 16, 2009.

\bibitem[Rog09]{rogers-5f4}
M.~D. Rogers.
\newblock New {$_5F_4$} hypergeometric transformations, three-variable {M}ahler
  measures, and formulas for $1/\pi$.
\newblock {\em Ramanujan Journal}, 18(3):327--340, 2009.

\bibitem[SC67]{sc67}
A.~Selberg and S.~Chowla.
\newblock On {E}pstein's zeta-function.
\newblock {\em J. Reine Angew. Math.}, 227:86--110, 1967.

\bibitem[Slo15]{sloane-oeis}
N.~J.~A. Sloane.
\newblock {The On-Line Encyclopedia of Integer Sequences}, 2015.
\newblock Published electronically at \url{http://oeis.org}.

\bibitem[Sta99]{stanley-ec2}
R.~P. Stanley.
\newblock {\em Enumerative Combinatorics}, volume~2.
\newblock Cambridge University Press, 1999.

\bibitem[Tit39]{titchmarsh-theoryoffunctions}
E.~Titchmarsh.
\newblock {\em The Theory of Functions}.
\newblock Oxford University Press, 2nd edition, 1939.

\bibitem[Ver04]{verrill}
H.~A. Verrill.
\newblock Sums of squares of binomial coefficients, with applications to
  {P}icard-{F}uchs equations.
\newblock Preprint, 2004.

\bibitem[Wan13]{wan-phd}
J.~G. Wan.
\newblock {\em Random walks, elliptic integrals and related constants}.
\newblock PhD thesis, University of Newcastle, 2013.

\bibitem[Wat41]{watson-bessel}
G.~N. Watson.
\newblock {\em A Treatise on the Theory of Bessel Functions}.
\newblock Cambridge University Press, 2nd edition, 1941.

\bibitem[Zei90]{zeilberger90}
D.~Zeilberger.
\newblock A holonomic systems approach to special function identities.
\newblock {\em Journal of Computational and Applied Mathematics},
  32(3):321--368, 1990.

\end{thebibliography}


\bigskip

\noindent Jonathan M. Borwein\\
\textsc{CARMA, University of Newcastle, Australia}\\
\texttt{jonathan.borwein@newcastle.edu.au}

\medskip

\noindent Armin Straub\\
\textsc{Department of Mathematics, University of Illinois at Urbana-Champaign, USA}\\
\texttt{astraub@illinois.edu}

\medskip

\noindent Christophe Vignat\\
\textsc{Department of Mathematics, Tulane University, USA}\\
\textsc{LSS, Supelec, Universit\'{e} Paris Sud, France}\\
\texttt{cvignat@tulane.edu}

\end{document}